\documentclass{amsart} 

\usepackage[active]{srcltx}

\newtheorem{definition}{Definition}[section]
\newtheorem{conj}[definition]{Conjecture}
\newtheorem{remark}[definition]{Remark}
\newtheorem{example}[definition]{Example}
\newenvironment{ack}{\noindent{\bf Acknowledgements}.}{}

\newtheorem{lemma}[definition]{Lemma}
\newtheorem{proposition}[definition]{Proposition}
\newtheorem{theorem}[definition]{Theorem}
\newtheorem{corollary}[definition]{Corollary}

\def\P{{\mathbb{P}}}

\def\K{{\mathbb{K}}}

\def\N{{\mathbb{N}}}

\def\I{{\mathcal{I}}}
\def\R{{\mathrm{R}}}

\def\ualpha{{\underline{\alpha}}}
\def\ubeta{{\underline{\beta}}}
\def\ugamma{{\underline{\gamma}}}
\def\udelta{{\underline{\delta}}}

\def\bU{{\mathbf{U}}}
\def\bV{{\mathbf{V}}}
\def\bX{{\mathbf{X}}}
\def\bg{{\mathbf{g}}}
\def\bp{{\mathbf{p}}}
\def\bh{{\mathbf{h}}}

\def\cS{{\mathcal{S}}}
\def\cE{{\mathcal{E}}}

\def\cF{{\mathcal{F}}}

\def\T{{\underline{T}}}

\def\X{{\underline{X}}}

\def\Z{{\mathbb{Z}}}

\def\bff{{\mathbf{f}}}

\def\bS{{\mathbf{S}}}
\def\bs{{\mathbf{s}}}

\def\S{{\mathrm{S}}}

\def\bdeg{{\mbox{bideg}}}

\def\cC{{\mathcal{C}}}

\def\e{{\bf{e}}}

\begin{document}
\title{The Rees Algebra of a monomial plane parametrization}

\author{Teresa Cortadellas Ben\'itez}
\address{Universitat de Barcelona, Facultat de Formaci\'o del Professorat.
Passeig de la Vall d'Hebron 171,
08035 Barcelona, Spain}
\email{terecortadellas@ub.edu}

\author{Carlos D'Andrea}
\address{Universitat de Barcelona, Facultat de Matem\`atiques.
Gran Via 585, 08007 Barcelona, Spain} \email{cdandrea@ub.edu}
\urladdr{http://atlas.mat.ub.es/personals/dandrea}
\thanks{Both authors are supported by the Research Project MTM2010--20279 from the
Ministerio de Ciencia e Innovaci\'on, Spain}


\begin{abstract}
We compute a minimal bigraded resolution of the Rees Algebra associated to a proper rational parametrization of a monomial plane curve. We describe explicitly both the bigraded Betti numbers and the maps of the resolution in terms of a generalized version of the Euclidean Algorithm. We also explore the relation between pencils of adjoints of the monomial plane curve and elements in a suitable piece of the defining ideal of the Rees Algebra.
\end{abstract}
\date{\today}
\maketitle

\section{Introduction}\label{zero}
In the last years, a lot of attention has been given to compute minimal generators of {\em moving curve ideals of rational parametrizations.} This is partially motivated by understanding the so-called method of implicitization of a rational parametrization by using  moving curves  stated by Sederberg and his collaborators in the 90's, see \cite{SC95, SGD97}. After the connection made by David Cox in \cite{cox08} between this problem and the computation of the defining ideal of the Rees Algebra associated to the parametrization, several  cases have been studied, see for instance \cite{CHW08,HSV08,bus09,HSV09,KPU09,HW10,CD10,HS12,CD13,CD13b,KPU13} and the references therein.
In this paper, we deal with the case of the  monomial plane curve, i.e. when the parametrization is given by a monomial map of the form
\begin{equation}\label{varphi}
\begin{array}{cccc}
\varphi:&\P^1_\K&\to&\P^2_\K\\
&(t_0:t_1)&\mapsto&\big(t_0^d:t_0^{d-u}t_1^u:t_1^d\big).
\end{array}
\end{equation}
Here, $\P^i_\K,\,i=1,2,$ denotes the projective space of dimension $i$ over an arbitrary field $\K,$ and $\gcd(d,u)=1.$ In this case, it is easy to see that the defining polynomial of the rational curve defined by \eqref{varphi} is  $X_1^d-X_0^{d-u}X_2^u,$ which is one of the elements in the Rees Algebra of the parametrization. However,  very little seems to be known about  other nontrivial elements of the Rees algebras of monomial curves. 
\par To show how our results work, we will exhibit them with detail on a particular case.  Set  $R=\K[T_0,T_1]$ and $\S=R[X_0,X_1,X_2]=\K[T_0,T_1,X_0,X_1,X_2]$, where $T_i,\,X_j$ are new variables $i=0,1,\,j=0,1,2,$ and denote with $\X=\{X_0,\,X_1\},\,\T=\{T_0,\,T_1,\,T_2\}$ for short.
Set now $d=10,\, u=3,$ so \eqref{varphi} becomes
$$
\begin{array}{ccc}
\P^1&\to &\P^2\\
(t_0:t_1)&\mapsto & (t_0^{10}:t_0^{7}t_1^3:t_1^{10}).
\end{array}
$$ 

Consider now the ideal $\I\subset R$ defined as $\I=\langle T_0^{10},\,T_0^{7}T_1^3,\,T_1^{10}\rangle.$ Its {\em Rees Algebra} is the ring
$\mbox{Rees}(\I)=\oplus_{n\geq0} \I^n\,Z^n,$ where $Z$ is a new variable. To study this ring in a more down-to-earth fashion, we consider the following epimorphism of $R$-modules:
$$
\begin{array}{cclcc}
R[\X]&\stackrel{\Phi_0}{\to}&\mbox{Rees}(\I)&\to&0\\
X_0&\mapsto&ZT_0^{10}\\
X_1&\mapsto&ZT_0^{7}T_1^3\\
X_2&\mapsto&ZT_1^{10}.
\end{array}
$$
The kernel of $\Phi_0$ is what it is known as the {\em ideal of moving curves} which follow $\varphi$, and the search for elements of minimal bi-degree in this ideal has been in the core of the so-called {\em method of moving curves for implicitization} studied in the nineties. Note that if we consider $\mbox{Rees}(\I)$ as a finitely generated $S$-module via this map, and declare that $\mbox{bideg}(Z)=(-10,1),\,\mbox{bideg}(T_i)=(1,0)$ and $\mbox{bideg}(X_j)=(0,1)$ for $i=0,1,\,j=0,\,1,\,2$ then $\Phi_0$ is  a bihomogeneous (of bidegree $(0,0)$)  $\S$-linear map.  
\par
For $ (a,b)\in \mathbb Z^2$ denote  with $\S(a,b) $ the twisted bigraded free module for which the bihomogeneous component of bidegree $(k,l) $ is defined as $\S(a,b)_{(k,l)}=\S_{(a+k,b+l)}$.
Our first main result of this paper, Theorem \ref{ttt}, states that the minimal bigraded free resolution of this module is the following
$$
0\to F_3 \stackrel{\Phi_3}{\to} F_2  \stackrel{\Phi_2}{\to} F_1  \stackrel{\Phi_1}{\to} F_0 \stackrel{\Phi_0}{\to} \mbox{Rees}(\I)\to 0,
$$
for suitable maps $\Phi_1,\,\Phi_2,\,\Phi_3,$ where
\[\begin{array}{ccl}
F_0&=& \S(0,0)\\
F_1&=& \S(0,-10)\oplus \S(-1,-3)\oplus \S(-1,-7)\oplus \S(-2,-4)\oplus \S(-3,-1)\oplus\S(-4,-2)\oplus \S(-7,-1),
\\
F_2&=& \S(-1,-10)^2\oplus \S(-2,-7)^2\oplus \S(-3,-4)^2\oplus \S(-4,-3)^2\oplus \S(-7,-2)^2,
\\
F_3&=&\S(-2,-10)\oplus \S(-3,-7)\oplus \S(-4,-4)\oplus \S(-7,-3). \end{array}
\]
which shows in particular that $\mbox{ker}(\Phi_0)$ has a minimal set of generators of $7$ elements. Moreover, we can make explicit the elements in each of the maps above. For instance,  via \eqref{fgen} one can compute the following $7$ elements in the kernel:
\begin{equation}\label{103}
\begin{array}{ccl}
F_{1,7}(\T,\X)&=&T_0^7X_2-T_1^7X_1\\
F_{2,4}(\T,\X)&=&T_0^4X_0X_2-T_1^4X_1^2\\
F_{3,3}(\T,\X)&=&T_0^3X_1-T_1^3X_0\\
F_{4,2}(\T,\X)&=& T_0^2X_1^4-T_1^2X_0^3X_2\\
F_{5,1}(\T,\X)&=& T_0X_1^7-T_1X_0^5X_2^2\\
F_{6,1}(\T,\X)&=& T_0X_0^2X_2-T_1X_1^3\\
F_{7,0}(\T,\X)&=& X_0^7X_2^3-X_1^{10}.
\end{array}
\end{equation}
By Theorem \ref{mgenerators} this family turns out to be a set of minimal generators of $\mbox{ker}(\Phi_0),$ and a reduced Gr\"obner basis of this ideal with respect to the lexicographic monomial order with  $X_2\prec X_1\prec X_0\prec T_1\prec T_0.$
Moreover, by denoting with $\{\e_1,\ldots, \e_7\}$ the canonical basis of $\S^7,$ where the vector $\e_i$ is associated to $F_{i,b_i}(\T,\X)$ above, $i=1,\ldots, 7,$ by Theorem \ref{phi1} and \eqref{2en1},\,\eqref{rrest}, we have that the following $10$ syzygies generate minimally $\mbox{ker}(\Phi_1)\subset\S^7$, and are a Gr\"obner basis of this submodule with respect to a suitable monomial order:
\[\begin{array}{l}
\bs_{1,2}=X_0\e_1-T_0^3\e_2-T_1^4X_1\e_3   \\
\bs_{1,3}=X_1\e_1-T_0^4X_2\e_3-T_1^3\e_2 \\
\bs_{2,3}=X_1\e_2-T_0X_0X_2\e_3-T_1^3\e_6\\
\bs_{2,6}=X_0\e_2-T_0^3\e_6-T_1X_1^2\e_3\\
\bs_{3,4}=X_1^3\e_3-T_0\e_4-T_1^2X_0\e_6 \\
\bs_{3,6}=X_0^2X_2\e_3-T_0^2X_1\e_6-T_1\e_4 \\
\bs_{4,5}=X_1^3\e_4-T_0\e_5-T_1X_0^3X_2\e_6 \\
\bs_{4,6}=X_0^2X_2\e_4-T_0X_1^4\e_6-T_1\e_5\\
\bs_{5,6}= X_0^2X_2\e_5-X_1^7\e_6+T_1\e_7\\
\bs_{6,7}=X_0^5X_2^2\e_6-T_0\e_7+X_1^3\e_5.
\end{array}
\]
We can also compute the minimal set of generators of  $\mbox{ker}(\Phi_2) \subset \S^{10}$ via Proposition \ref{aiecomae}. They have the form
$$
\begin{array}{l}
\bs_{1,2,3}=X_1\e_{1,2}-X_0\e_{1,3}+T_0^3\e_{2,3}-T_1^3\e_{2,6},\\
\bs_{2,3,6}=X_0\e_{2,3}-X_1\e_{2,6}+T_0\e_{3,6}-T_1\e_{3,4},\\
\bs_{3,6,3}=X_0^2X_2\e_{3,4}-X_1^3\e_{3,6}+T_0\e_{4,6}-T_1\e_{4,5},\\
\bs_{4,5,6}=X_0^2X_2\e_{4,5}-X_1^3\e_{4,6}+T_0\e_{5,6}-T_1\e_{5,7}.
\end{array}
$$
where $\{\e_{i,j}\}_{i, j}$ denotes the canonical basis of $\S^{10},$ indexed by the rules given in \eqref{ee}. Theorem \ref{phi2} then states that this not only a set of minimal generators of this submodule, but also a Gr\"obner basis of it.
\par
We will see in the text that to make the bidegrees  explicit in full detail,  we must not only consider the classical Euclidean Remainder Sequence applied to  $(d,u),$ but also some sort of {\em Slow Euclidean Remainder Sequence} (SERS) which is worked out in detail in Section \ref{s2}. For instance, the standard Euclidean Remainder Sequence associated to $(10,3)$ is  
$$\begin{array}{ccl}
7&=&2\cdot 3+1\\
3&=&3\cdot 1+0,
\end{array}
$$
and from these numbers one can already read the exponents appearing in the minimal resolution of $\mbox{Rees}(\I)$ (see \eqref{uno} and \eqref{q} in Section \ref{setup}). To work out the generators in each step, one should look at the numbers appearing in the  SERS, and also at the {\em B\'ezout type} identities associated to them. For  $(10,3),$ its SERS  can be recovered from
$$\begin{array}{ccl}
7&=&2\cdot 3+1\\
4&=&1\cdot 3+1\\
3&=&2\cdot 1+1\\
2&=&1\cdot 1+1\\
1&=& 1\cdot 1+0.
\end{array}
$$
 Very little seems to be known about the description of minimal generators in the resolution of the Rees Algebra of monomial parametrizations in general. In \cite{MS10}, an explicit set of generators of $\mbox{ker}(\Phi_0)$ is found in the case of an affine monomial curve in a four-dimensional space, with some strong conditions on the exponents of the parametrization.  Some results have also been obtained for square-free monomial maps, see for instance \cite{vil08, GRV09}, but the general picture yet seems to be unknown. In this sense,
our results can be regarded as another step towards understanding the combinatorics of the Rees Algebra of monomial curves.
\par  It should be mentioned, however, that $\mbox{ker}(\Phi_0)$  is an example of a codimension $2$ lattice ideal, as presented in \cite{PS98}. In that paper, a minimal resolution of these ideals is given in geometric terms, by using lattice free polytopes built from the monomial map $\Phi_0.$  In particular, the fact that the resolution of $\mbox{Rees}(\I)$ has length $3$ in this case is a direct consequence of Theorem $2.34$ in \cite{PS98}.  However, our approach is completely different in the sense that no geometry is involved in our calculations at all, and we can give the whole resolution by applying very simple arithmetics on the initial data $(d,u).$
\par
The paper is organized as follows: in Section \ref{setup} we properly state the basic notation and main results. Then we move to Section \ref{s2}, where we study properties of Slow Extended Euclidean Remainder Sequences which will be useful for the proof of the main results. In Section \ref{GB} we recall basic facts and properties of Gr\"obner bases of submodules of $\S^m$ and syzygies. Theorems \ref{mgenerators}, \ref{phi1} and \ref{phi2} are proven in Sections \ref{minimal}, \ref{minimal1} and \ref{minimal2} respectively. 

In Section \ref{adj} we turn into the study of geometric elements associated to the parametric monomial curve, and a connection between elements  of $\T$-degree $1$ in $\mbox{ker}(\Phi_0),$ and {\em pencils of adjoints} associated to the curve $\cC_{\mu,d}\subset\P^2_\K,$ which is the image of $\varphi$ in \eqref{varphi}. This connection has been observed already by Cox  in \cite{cox08}, and some conjectures were posted at that time. In Theorem \ref{adjj}, we compute explicitly the dimension of the $\K$-vector space of those forms in   $\mbox{ker}(\Phi_0)_{(1,\ell)}$ which happen to be elements of $\mbox{Adj}_\ell(\cC_{\mu,d}),$ the $\K$-vector space of pencils of adjoints of $\cC_{\mu,d}$ having degree $\ell,$  and  we measure how different they are  by computing explicitly
$\dim_\K\left(\mbox{\rm ker}(\Phi_0)_{1,\ell}\, /\,\mbox{\rm Adj}_\ell(\cC_{\mu,d})\cap\mbox{\rm ker}(\Phi_0)_{1,\ell} \right)$ for $\ell\geq d-2$ in
Theorem \ref{kkk}.
The paper concludes with some further examples, the last of them showing that the bounds given in Theorem \ref{kkk} in Section \ref{examples}.

\smallskip
\begin{ack}
We are grateful to Eduardo Casas-Alvero for several discussions on blow-ups and adjoint curves,  to David Cox for having posed us interesting questions on a preliminary version of this draft, and to the anonymous referees for very useful suggestions for improving the presentation of our work. All our computations and experiments were done with the aid of the softwares {\tt Macaulay 2} \cite{mac} and {\tt Mathematica} \cite{math}.

\end{ack}

\section{Statement of the Main Results}\label{setup}
With notation as above, set $u,\,d$ be  positive integers with $u<\frac{d}2$ and  $\gcd(u,d)=1.$ Consider the homogeneous ideal
$\I=\langle T_0^d,\,T_0^{d-u}T_1^u,\,T_1^d\rangle\subset\R,$ and set
$\mbox{Rees}(\I)=\oplus_{n\geq0}\I^nZ^n$ for the Rees Algebra associated to $\I$.  There is an epimorphism of $\K$-algebras defined by
\begin{equation}\label{presentacion}
\begin{array}{cclcc}
R[\X]&\stackrel{\Phi_0}{\to}&\mbox{Rees}(\I)&\to&0\\
T_0&\mapsto&T_0\\
T_1&\mapsto&T_1\\
X_0&\mapsto&ZT_0^d\\
X_1&\mapsto&ZT_0^{d-u}T_1^u\\
X_2&\mapsto&ZT_1^d.
\end{array}
\end{equation}

We consider $\mbox{Rees}(\I)$ as a finitely generated $S$-module via \eqref{presentacion}. If we declare that $\mbox{bideg}(Z)=(-d,1),\,\mbox{bideg}(T_i)=(1,0)$ and $\mbox{bideg}(X_j)=(0,1)$ for $i=0,1,\,j=0,\,1,\,2$ then $\Phi_0$ is  a bihomogeneous (of bidegree $(0,0)$)  $\S$-linear map. The main goal of this article is to present a minimal bigraded free resolution of $\mbox{Rees}(\I)$.
Indeed, if we do not keep track of the graduation, we will show that such a minimal resolution is of the form
\begin{equation}\label{uno}
0\to \S^{q-1}\to \S^{2q} \to \S^{q+2} \to \S \stackrel{\Phi_0}{\to} \mbox{Rees}(\I)\to 0,
\end{equation}
for a suitable positive integer $q$ to be determined in the sequel.
\par
Our main result is that by performing some simple arithmetics on the pair $(d,u)$ we can actually make explicit the whole minimal resolution of $\mbox{Rees}(\I)$ without the need of any geometrical or homological construction as it was done for instance in \cite{PS98}. To  compute  the number $q$ appearing in \eqref{uno}, we proceed as follows: consider the standard {\em Euclidean Remainder Sequence}  $\{a_n\}_{n=0,\ldots, p},$ $\{q_{m}\}_{m=1,\ldots, p-1},$  associated to the data $(d,u)$ which is defined as follows:
$a_0=d-u,\,a_1=u.$
And for $1\leq i\leq p-1$ we write $a_{i-1}=q_{i}a_{i}+a_{i+1},$ with $q_i,\,a_{i+1}\in\Z,\,0\leq a_{i+1}<a_{i}.$
The number $p=p(d,u)$ is such that $a_p=0.$  The exponent $q$ in \eqref{uno} is now defined as 
\begin{equation}\label{q}
q=\sum_{m=1}^{p-1}q_m.
\end{equation}
 To precise all the other maps in the resolution,  we have to take a close look at the Euclidean Algorithm, and consider the {\em Slow Euclidean Remainder Sequence} (SERS) associated to  $(d,u),$ which are pairs of nonegative integers $\{(b_n,\,c_n)\}_{n=1,\ldots, q+1}$  defined recursively as follows:
$b_1=d-u,\,c_1=u,$ and for $0\leq n\leq q,$ the set $\{b_{n+1},\,c_{n+1}\}$ is equal to $\{b_n-c_n,\,c_n\},$ sorted in such a way that
$b_{n+1}\geq c_{n+1}.$
\par
It is easy to see that  the SERS can be regarded as a way of performing the standard Euclidean Remainder Sequence without making any divisions.   We  will show in Section \ref{s2}  that, for any $n=1,\ldots, q+1,$  there is a standard way of writing
\begin{equation}\label{bezz}
\sigma_n u+\tau_n(d-u)=b_n,
\end{equation}  with $\sigma_n,\,\tau_n\in\Z,\, |\sigma_n|< d-u,\,|\tau_n|<u$ in the same way one unravels the Euclidean Remainder sequence to produce  B\'ezout identities associated to $u$ and $d-u.$
\par
For  $\ell=0,\,1,\dots ,p-1$, set $m_\ell=1+\sum_{j=1}^{\ell}q_j.$ Also set $m_{p}:=q+2$. For instance, in the case $(d,u)=(10, 3)$ of the introduction, we have
$m_0=1,\ m_1=3,\ m_2=6,$ and $m_3=7.$
Given $n\in\{1,\ldots, q\},$ we define $\ell(n)$ as the unique $\ell$ such that $m_{\ell-1} \leq n < m_\ell.$
 One of the main results of this paper is the following:
\begin{theorem}\label{ttt}
The minimal bigraded free resolution of $\mbox{\em Rees}(\I)$ is:
\begin{equation}\label{ddos}
\begin{array}{l}
0\to  \displaystyle {\oplus_{n=1}^{q-1}}\S(-(b_n, |\sigma_{n}-\tau_{n}|+ 2|\sigma_{m_{\ell(n)}}-\tau_{m_{\ell(n)}}| )) \to\\ \\ \stackrel{\Phi_3}{\to}
 \displaystyle {\oplus_{n=1}^{q}}\S(-(b_n, |\sigma_{n}-\tau_{n}|+ |\sigma_{m_{\ell(n)}}-\tau_{m_{\ell(n)}}| )) ^2
 \stackrel{\Phi_2}{\to} \displaystyle{\oplus_{n=1}^{q+2}}\S(-(b_n, |\sigma_n-\tau_n|))\to \\ \\
\stackrel{\Phi_1}{\to} \S \stackrel{\Phi_0}{\to} \mbox{\em Rees}(\I)\to 0.
 \end{array}
\end{equation}
\end{theorem}
The proof of this Theorem follows straightforwardly from Theorems \ref{mgenerators}, \ref{phi1} and \ref{phi2} below, where we also make the maps $\Phi_i,\,i=1,2,3,$ explicit. This is due to the well-known fact that knowing minimal generators of each of the syzygy modules leads to a minimal resolution of $\mbox{Rees}(\I),$  see for instance \cite[Chapter 6, Proposition 3.10]{CLO98}.  As a consequence of this result,  one
 can compute  the whole list of  bigraded Betti numbers of $\mbox{Rees}(\I) $ in terms of the SERS.

It should not be surprising to have numbers appearing from the Euclidean sequence between $d$ and $u$ in the resolution of $\mbox{Rees}(\I),$ as it is well-known that  the projective scheme defined by the Rees Algebra of $\I$ is the blowing-up of the spectrum of $\K[T_0,T_1]$ along the subscheme defined by this ideal, and hence the multiplicities of all the points of the monomial curve should play a role in its description. As it is shown in \cite[Theorem 8.4.12]{BK86}, to compute the multiplicity sequence of a monomial plane curve singularity one must deal with Euclidean sequences involving the exponents appearing in the monomial expansion of the parametrization.  Theorem \ref{ttt} essentially states that a finer algorithm than the classical Euclidean remainder is needed in order to get a full understanding of $\mbox{Rees}(\I).$

The SERS will also allow us to compute the maps $\Phi_j,\,j=1,\,2,\, 3$ in  \eqref{ddos} as follows.  We start by setting, using the notation introduced in \eqref{bezz}, for $n=1,\ldots, q+1,$
\begin{equation}\label{fgen}
F_{n, b_n}(\T,\X)=\left\{\begin{array}{rcl}
T_0^{b_n}X_0^{-\sigma_n}X_2^{\tau_n}-T_1^{b_n}X_1^{\tau_n-\sigma_n}& \,\mbox{if}&\, \sigma_n\leq0,\\
T_0^{b_n}X_1^{\sigma_n-\tau_n}-T_1^{b_n}X_0^{\sigma_n}X_2^{-\tau_n}& \,\mbox{if}&\, \sigma_n>0.
\end{array}
\right.
\end{equation}

We define also $b_{q+2}=0$ and
$$F_{q+2,0}(\T,\X)=\left\{\begin{array}{ccl}
X_0^{d-u}X_2^u-X_1^d&\,\mbox{if}\,&\sigma_q>0\\
X_1^d-X_0^{d-u}X_2^u&\,\mbox{if}\,&\sigma_q< 0.
\end{array}\right.$$
 By using \eqref{bezz}, we easily verify that
$$F_{n,b_n}(T_0,T_1,ZT_0^d,\,ZT_0^{d-u}T_1^u, ZT_1^d)=0,\ \ n=1,\ldots, q+2,$$ i.e. all these elements belong to $\mbox{ker}(\Phi_0)$. Set $\cF_0:=\{F_{n, b_n}(\T,\X)\}_{n=1,\ldots, q+2}.$ We will see in Remark \ref{seqq} that the elements of $\cF_0$ can be defined recursively, without having to compute all  the B\'ezout's identities \eqref{bezz}. In fact,  one may regard this sequence as a generalization of the process we have set in \cite{CD13b} to produce minimal elements in $\mbox{ker}(\Phi_0).$

Recall that for a given monomial order  in $\S^k$, a {\em Gr\"obner basis} of an $\S$- submodule $M\subset\S$ is a set of generators of $M$ such that their leading terms with respect to this monomial order generate the {\em initial module} $\mbox{lt}(M),$ see Section \ref{GB} for more background on these concepts. We will denote with $\prec_l$ the lexicographic order on the monomials of $\S$ such that $X_2\prec_l X_1\prec_l X_0\prec_l T_1\prec_l T_0$ if $\sigma_q>0,$ or $X_2\prec_l X_0\prec_l X_1\prec_l T_1\prec_l T_0$ if $\sigma_q\leq0.$
\par
Consider the free module $\displaystyle {\bigoplus_{n=1}^{q+2}}\S(-(b_n, |\sigma_n-\tau_n|))$ with basis $\e_1,\dots,\e_{q+2}$ and  $\bdeg(\e_n)=(b_n, | \sigma_n-\tau_n|)$. We will identify $\e_n$ with $F_{n,b_n}(\T,\X),$ note that straightforwardly we have $\bdeg(F_{n,b_n}(\T,\X))=(b_n, | \sigma_n-\tau_n|)$ as well. The following result will be proven at the end of Section \ref{minimal}.
\begin{theorem}\label{mgenerators}
The map $\Phi_1$ in \eqref{ddos} can be made explicit as follows:
\[\begin{array}{ccc}
 \displaystyle{\bigoplus_{n=1}^{q+2}}\S(-(b_n, |\sigma_n-\tau_n|))&\stackrel{\Phi_1}{\longrightarrow}& \S \\
\e_n&\mapsto&F_{n,b_n}(\T,\X).
\end{array}
\]
Moreover, $\cF_0$ is a minimal set of generators of $\mbox{\rm ker}(\Phi_0)$, and  a reduced Gr\"obner basis of this ideal with respect to $\prec_l.$
\end{theorem}

Related to this result, we mention \cite[Proposition 8.3]{PS98}, where it is stated that the reduced Gr\"obner basis of $\mbox{\rm ker}(\Phi_0)$ with respect to some reverse lexicographic term order is actually a minimal generating set.

To explicit  $\Phi_2,$ in Section \ref{GB} we will endow the free  module $\S^{q+2}$ with a term ordering $\prec_{l,\cF_0}$ depending on both $\cF_0$ and the term order $\prec_{l}$ on $\S.$ We will then identify $2q$ specific syzygies on the elements of $\cF_0.$
To make this more precise, for $n=1,\ldots,q,$ and $\ell=\ell(n),$ we define
 \begin{equation}\label{rho} \rho(n)=\left\{\begin{array}{ccl}
n+1&\,\mbox{ if }\,&n+1<m_\ell\\
m_{\ell+1}&\,\mbox{ if }\,&n+1=m_\ell.
\end{array}\right.\end{equation}
Observe that  we always have $b_{m_{\ell(n)}}=c_n$ and $b_{\rho(n)}=b_n-c_n$.
By computing the ${\bf S}$-polynomials
 ${\bf S}\big(F_{n,b_n}(\T,\X), F_{m_{\ell(n)}, b_{m_{\ell(n)}}}(\T,\X)\big), \,
 {\bf S}\big(F_{n,b_n}(\T,\X), F_{\rho(n), b_{\rho(n)}}(\T,\X)\big)$
 in Lemmas \ref{szg} and \ref{extra},
we obtain the following elements in $\mbox{syz}(\cF_0)\subset\S^{q+2}$, the {\em syzygy module} of the family $\cF_0$, which is isomorphic to $\mbox{ker}(\Phi_1)$:
\begin{equation}\label{2en1}
\begin{array}{ccl}
\bs_{n,\rho(n)}&=&\left\{\begin{array}{lcr}
X_0^{\sigma_{m_{\ell(n)}}}X_2^{-\tau_{m_{\ell(n)}}}\e_n-T_0^{b_{m_{\ell(n)}}}\e_{\rho(n)}-T_1^{b_{\rho(n)}}X_1^{-\sigma_n+\tau_n}\e_{m_{\ell(n)}}&\,\mbox{if}&\, \sigma_n\leq0,\\
X_1^{\tau_{m_{\ell(n)}}-\sigma_{m_{\ell(n)}}}\e_n-T_0^{b_{m_{\ell(n)}}}\e_{\rho(n)}-T_1^{b_{\rho(n)}}X_0^{\sigma_n}X_2^{-\tau_n}\e_{m_{\ell(n)}}
&\,\mbox{if}&\, \sigma_n>0,\\
\end{array}
\right.\\ \\
\bs_{n,m_{\ell(n)}}&=&\left\{\begin{array}{lcr}
X_1^{\sigma_{m_{\ell(n)}}-\tau_{m_{\ell(n)}}}\e_n-T_0^{b_{\rho(n)}}X_0^{-\sigma_n}X_2^{\tau_n}\e_{m_{\ell(n)}}-T_1^{b_{m_{\ell(n)}}}\e_{\rho(n)}
 &\,\mbox{if}&\, \sigma_n\leq0,\\
X_0^{-\sigma_{m_{\ell(n)}}}X_2^{\tau_{m_{\ell(n)}}}\e_n-T_0^{b_{\rho(n)}}X_1^{\sigma_n-\tau_n}\e_{m_{\ell(n)}}-T_1^{b_{m_{\ell(n)}}}\e_{\rho(n)}&\,\mbox{if}&\,\sigma_n>0,
\end{array}
\right.
\end{array}
\end{equation}
for $n\leq q,$
and
\begin{equation}\label{rrest}
\bs_{q+1,q+2}=\left\{\begin{array}{lcl}
X_0^{-\sigma_{q}}X_2^{\tau_{q}}\e_{q+1}-T_0\e_{q+2}-X_1^{\tau_{q+1}-\sigma_{q+1}}\e_{q} &\,\mbox{if}\,&\sigma_{q+1}\leq 0,\\
X_1^{-\sigma_q+\tau_q}\e_{q+1}-T_0\e_{q+2}-X_0^{\sigma_{q+1}}X_2^{-\tau_{q+1}}\e_{q},&\,\mbox{if}\,&\sigma_{q+1}> 0.
\end{array}
\right.
\end{equation}
Set now
$\cF_1:= \left\{\bs_{q,q+1},\,  \bs_{q+1,q+2}\right\}\cup\left\{ \bs_{n,m_{\ell(n)}},\, \bs_{n,\rho(n)}
 \right\}_{n=1,\ldots, q-1}.$
Observe that this set has  $2q$ bihomogeneous elements with
\begin{equation}\label{25252}
\begin{array}{l}
 \bdeg(\bs_{n,m_{\ell(n)}})=\bdeg(\bs_{n,\rho(n)})=(b_n, |\sigma_n-\tau_n|+|\sigma_{m_{\ell(n)}}-\tau_{m_{\ell(n)}}|)=(b_n, |\sigma_{\rho(n)}-\tau_{\rho(n)}|),\\
  \bdeg(\bs_{q,q+1})=\bdeg(\bs_{q,q+2})=(1,d)=(b_q, |\sigma_{\rho(q)}-\tau_{\rho(q)}|).
\end{array}
\end{equation}
This lead us to consider the free bigraded module
$\displaystyle {\oplus_{n=1}^{q}}\S(-(b_n, |\sigma_{\rho(n)}-\tau_{\rho(n)}|  )^2,$ having canonical basis
\begin{equation}\label{ee}
\{\e_{n,\rho(n)},\,\e_{n,m_{\ell(n)}}\}_{n=1,\ldots, q-1}\bigcup\{\e_{q,q+1}\e_{q+1,q+2}\},
\end{equation}
and declaring that $\bdeg(\e_{n,\rho(n)})=\bdeg(\e_{n,m_{\ell(n)}})=(b_n, | \sigma_{\rho(n)}-\tau_{\rho(n)}| )$, and
 $\bdeg(\e_{q,q+1})=\bdeg(\e_{q+1,q+2})=(1,d).$

\begin{theorem}\label{phi1}
The map $\Phi_2$ in \eqref{ddos} can be made explicit as follows:
\[\begin{array}{ccl}
 \displaystyle {\bigoplus_{n=1}^{q}}\S(-(b_n, |\sigma_{\rho(n)}-\tau_{\rho(n)}| )) ^2 &\stackrel{\Phi_2}{\to}&
 \displaystyle{\bigoplus_{n=1}^{q+2}}\S(-(b_n, |\sigma_n-\tau_n|)) \\
 \e_{n,k}&\mapsto& \bs_{n,k}.
 \end{array}
\]
Moreover, $\cF_1$ is a reduced Gr\"obner basis of $\mbox{\em ker}(\Phi_1)$ with respect to $\prec_{l,\cF_0},$ and a minimal set of generators of this module.
\end{theorem}
To complete our descripton of the resolution in \eqref{ddos},  we have to explicit $\Phi_3.$ In Lemma \ref{extra2}, we will see that for each $n=1,\ldots, q,$ with the induced order $\prec_{l,\cF_1},$
$$\begin{array}{rcl}
\bS\big(\bs_{n,\rho(n)}, \bs_{n,m_{\ell(n)}}\big)&=&F_{\rho(n),b_{\rho(n)}}(\T,\X)\e_{m_{\ell(n)}}-F_{m_{\ell(n)},b_{m_{\ell(n)}}}(\T,\X)\,\e_{\rho(n)},\\
T_1^{b_{m_{\ell(n)}}}\bs_{n,\rho(n)}-T_0^{b_{m_{\ell(n)}}}\bs_{n,m_{\ell(n)}}&=&F_{n,b_n}(\T,\X)\e_{m_{\ell(n)}}-F_{m_{\ell(n)},b_{m_{\ell(n)}}}(\T,\X)\e_n.
\end{array}
$$
These equalities will help us, in Proposition \ref{aiecomae},  write the syzygies $\bs_{(n,\rho(n)),(n,m_{\ell(n)})}\in\mbox{syz}(\cF_1)\subset\S^{2q}$ in an explicit way. We will denote these syzygies with $\bs_{n,\rho(n),\ell(n)}$ for short.
for $n=1,\ldots q.$  We will see also that:
\begin{equation}\label{consistent2}
  \bdeg (\bs_{n,\rho(n),\ell(n)})=(b_n, |\sigma_n-\tau_n|+|\sigma_{m_{\ell(n)}}-\tau_{m_{\ell(n)}}|), \
\ n=1,\ldots q,\end{equation}
which leads us to consider the module  $\displaystyle {\oplus_{n=1}^{q-1}}\S(-(b_n, |\sigma_{n}-\tau_{n}|+ 2|\sigma_{m_{\ell(n)}}-\tau_{m_{\ell(n)}}| ))$
where we denote its canonical basis with $\{\e_{n, \rho(n),\ell(n)}\}_{n=1,\ldots, q-1},$ and declare that
$$\bdeg (\e_{(n,\rho(n),\ell(n)})=(b_n, |\sigma_n-\tau_n|+|\sigma_{m_{\ell(n)}}-\tau_{m_{\ell(n)}}|), \
\ n=1,\ldots q.$$ Set $\cF_2:=\{\bs_{n,\rho(n),\ell(n)}\}_{n=1,\ldots, q-1} .$  In Section \ref{minimal2}, we will show the following:
\begin{theorem}\label{phi2}
The map $\Phi_3$ in \eqref{ddos} can be made explicit as follows:
$$
\begin{array}{ccc}
\displaystyle {\oplus_{n=1}^{q-1}}\S(-(b_n, |\sigma_{n}-\tau_{n}|+ 2|\sigma_{m_{\ell(n)}}-\tau_{m_{\ell(n)}}| )) &\stackrel{\Phi_3}{\to}&
 \displaystyle {\oplus_{n=1}^{q}}\S(-(b_n, |\sigma_{n}-\tau_{n}|+ |\sigma_{m_{\ell(n)}}-\tau_{m_{\ell(n)}}| )) ^2\\
 \e_{n,\rho(n),\ell(n)}&\mapsto& \bs_{n,\rho(n),\ell(n)}.
 \end{array}
$$
Moreover, $\cF_2$ is a Gr\"obner basis of $\mbox{\rm ker}(\Phi_2)$ with respect to $\prec_{l,\cF_1}.$ This set is also  $S$-linearly independent (in particular, a minimal set of generators of this module).
\end{theorem}

In Section \ref{adj}, we will focus on geometric features of the monomial curve $\cC_{u,d}$ which is the image of \eqref{varphi}, and its connections with elements of $\T$-degree one in $\mbox{ker}(\Phi_0).$ The exploration of this kind of relations was suggested by David Cox in \cite{cox08}, and some partial studies over specific families of curves have been obtained in \cite{bus09, CD13b}.  We will focus there on the monomial plane curve case. As we are going to use standard tools of Algebraic Geometry designed for curves over the complex numbers,  all along that section we will assume that $\K$ is an algebraically closed field of characteristic zero.  In a rough way,  a curve $\tilde{\cC}$ is adjoint to another curve $\cC$ if for any point ${\bf p}\in\cC$, including  those of {\em virtual} nature, we have
\begin{equation}\label{isi}
m_{\bf p}(\tilde{\cC})\geq m_{\bf p}(\cC)-1.
\end{equation}
Here, $m_{\bf p}(\cC)$ denotes the multiplicity of ${\bf p}$ with respect to $\cC$.  This definition is not precise at all, and we refer the reader to \cite[Sections 4.1 and 4.8]{cas00} for a correct statement of this concept. Adjoint curves are of importance in computational algebra due to their use in the inverse of the implicitization problem, see for instance \cite{SWP08} and the references therein.

\begin{definition}
A pencil of adjoints of  $\cC$ of degree $\ell\in\N$ is a bihomogeneous form $T_0 C^0_\ell(\X)+T_1C^1_\ell(\X)\in\S,$ with $C^i_\ell(\X)$ of degree $\ell$, defining (scheme-theoretically) a curve $\cC^i_\ell$ adjoint of $\cC,$ for  $i=0,1.$
\end{definition}
For $\ell\in\N,$ we denote with $\mbox{Adj}_\ell(\cC)$ the $\K$-vector space of pencils of adjoints of $\cC$ of degree $\ell$.
\par
In \cite[Conjecture 3.8]{cox08}, it was conjectured that for $u>1,$ a set of minimal generators of $\mbox{ker}(\Phi_0)$ of bidegree $(1,m)$, with $m\in\{d-1,\,d-2\},$ can be chosen to be pencils of adjoints. This conjecture was shown to hold for curves with ``mild'' multiplicities (see \cite[Corollaries 4.10 \& 4.11]{bus09}), but fails in general, see for instance \cite{CD13b}. We will show also in Theorem \ref{kkk}  below, that in the monomial plane curve, the conjecture does not hold either.

Let $\nu_{u,d}$ be the number of solutions of $\ualpha=(\alpha_0, \alpha_1, \alpha_2)\in\N^3$ such that $|\ualpha|=\ell-|\sigma_q-\tau_q|,$ with $\ell\geq d-2,$ satisfying

\begin{equation}\label{pp1}
\begin{array}{ccl}
u\alpha_1+d\alpha_2&<& (d-1)(u-1)-d|\tau_q|, \\  & \mbox{or}&  \\
d\alpha_0+(d-u)\alpha_1&<&(d-1)(d-u-1)-(d-u)|\sigma_q-\tau_q|,
\end{array}
\end{equation}
 plus  the number of solutions of $\ubeta=(\beta_0, \beta_1, \beta_2)\in\N^3$ such that $|\ubeta|=\ell-|\sigma_{q+1}-\tau_{q+1}|$ with $\ell\geq d-2,$ satisfying 
\begin{equation}\label{pp2}
\begin{array}{ccl}
u\beta_1+d\beta_2&<& (d-1)(u-1)-u|\sigma_{q+1}-\tau_{q+1}|, \\  & \mbox{or} & \\
d\beta_0+(d-u)\beta_1&<&(d-1)(d-u-1)-d|\sigma_{q+1}|.
\end{array}
\end{equation}

Note that $\nu_{u,d}$ does not depend on $\ell.$ In Section \ref{adj}, we will prove both Lemma \ref{ddim} and Theorem \ref{adjj}, from which one deduces straightforwardly the following result.
\begin{theorem}\label{kkk} For $\ell\geq d-2,$ we have
$$\dim_\K\left(\mbox{\rm ker}(\Phi_0)_{(1,\ell)}\, /\,\mbox{\rm Adj}_\ell(\cC_{u,d})\cap\mbox{\rm ker}(\Phi_0)_{(1,\ell)} \right)=\nu_{u,d}.
$$
\end{theorem}
For $(d,u)=(10,3),$  we have that \eqref{pp1} turns into
$$
3\alpha_1+10\alpha_2<-2  \quad \mbox{or} \quad 
10\alpha_0+7\alpha_1<5,
$$
with $(\alpha_0,\alpha_1,\alpha_2)\in\N^3,\,\alpha_0+\alpha_1+\alpha_2=\ell-7.$
So, the only solution to this system of inequalities is actually $(0,0,\ell-7),$ for $\ell\geq7.$ On the other hand, we get that \eqref{pp2} turns into
$$ 3\beta_1+10\beta_2<9\quad \mbox{or} \quad
10\beta_0+7\beta_1<34.
$$
It is easy to see that there are $4$ solutions to the inequality on the left-hand side, namely one per each of the following values of $3\beta_1+10\beta_2:\, 0,\,3,\,6,\,9.$ For the second inequality, by computing straightforwardly one gets that the values of $10\beta_0+7\beta_1$ attainable with $\beta_0,\,\beta_1\in\N$ are the following twelve:
$$0,\,7,\,10,\,14,\,17,\,20,\,21,\,24,\,27,\,28,\,30,\,31.
$$
So, we have that $\nu_{10,3}=1+4+12=17.$
\par 
It is interesting to remark that the dimension of the quotient is independent of $\ell$ for $\ell\geq d-2,$ which is a situation that we already encountered in the case of $u=2$ with a point of very high multiplicity, see  \cite[Remark 4.5]{CD13b}. We wonder if this situation holds in general. To be more precise, we state the following conjecture.
\begin{conj}
Replace $T_0^d,\,T_0^{d-u}T_1^u,\,T_1^d$ in \eqref{presentacion} with polynomials \newline $u_0(\T),\,u_1(\T),\,u_2(\T)\in\K[\T]$ for  $i=0,1,2,$ homogeneous of degree $d$ and without common factors, such that they define a birational parametrization of a plane algebraic curve $\cC\subset\P^2.$  Then for $\ell\geq d-2,$
$\dim_\K\left(\mbox{\rm ker}(\Phi_0)_{(1,\ell)}\, /\,\mbox{\rm Adj}_\ell(\cC)\cap\mbox{\rm ker}(\Phi_0)_{(1,\ell)} \right)$ does not depend on $\ell.$
\end{conj}
We will see in Proposition \ref{arriba} a bound which is quadratic in $d$ for the value of $\nu_{\mu,d},$ and show in Section \ref{examples} that the quadratic nature of this bound is unavoidable.

\bigskip
\section{Euclidean Remainder Sequences }\label{s2}
As in Section \ref{setup}, consider the standard Euclidean Remainder Sequence  $\{a_n\}_{n=0,\ldots, p},$ $\{q_{m}\}_{m=1,\ldots, p-1},$  associated to the data $(d,u)$ which is defined as follows:
$a_0=d-u,\,a_1=u.$
And for $1\leq i\leq p-1$ we write $a_{i-1}=q_{i}a_{i}+a_{i+1},$ with $q_i,\,a_{i+1}\in\Z,\,0\leq a_{i+1}<a_{i}.$
The number $p=p(d,u)$ is such that $a_p=0.$
\par Associated with this well-known mathematical object, we define the so called {\em Extended Euclidean Remainder Sequence} $\{(s_n,\,t_n)\}_{n=0,\ldots, p}$ as follows:
$$\begin{array}{lcl}
s_0=0&\,&t_0=1\\
s_1=1&\,&t_1=0\\
s_{i+1}=s_{i-1}-q_i\,s_i&\,&t_{i+1}=t_{i-1}-q_it_i,\, i=1,\ldots, p-1.
\end{array}
$$
The following lemma collects some properties of these sequences.
\begin{lemma}\label{corr}$^{}$
\begin{enumerate}
\item[i)] $s_iu+t_i(d-u)=a_i$ for $0\leq i\leq p.$
\item[ii)] For all admissible values of $i\geq0,$ both $s_{2i},\,t_{2i+1}$ are either zero or negative integers, and both $s_{2i+1},\, t_{2i}$ are nonnegative.
\item[iii)]  $|s_0|<|s_1|\leq |s_2|<|s_3|<\ldots$ and also $ |t_1|< |t_2|\leq|t_3|<|t_4|<\ldots$
\item[iv)] For $i=1,\ldots p,\,|s_i|\leq \frac{d-u}{a_{i-1}},\, |t_i|\leq \frac{u}{a_{i-1}}.$
In particular, $|s_i|<d-u,$  and $|t_i|<u$ for all $i\leq p-1.$
\end{enumerate}
\end{lemma}

\begin{proof}
These results are classical. See for instance \cite[Lemma 3.8,\, Lemma 3.12 \& Exercise 3.15]{vzGG03}.
\end{proof}
\smallskip
For example, in the case $u=3,\,d=10,$ we have that
$$\{(s_n,\,t_n)\}_{n=0,1,2,3}=\{(0,1),\,(1,0),\,(-2,1),(7,-3)\}.
$$
We now recall the Slow Euclidean Remainder Sequence (SERS) from the introduction: it is a sequence of pairs $\{(b_n,\,c_n)\}_{n=1,\ldots, q+1}$  defined recursively as follows:
$b_1=d-u,\,c_1=u,$ and for $0\leq n\leq q,$ $(b_{n+1},\,c_{n+1})$ is univocally defined in such a way that $\{b_{n+1},\,c_{n+1}\}=\{b_n-c_n,\,c_n\},$ and
$b_{n+1}\geq c_{n+1}.$ Note that we straightforwardly have that $\{b_n\}$ is a decreasing sequence of nonegative integer numbers.
We will also consider a  sort of {\em Extended SERS}, which will be a sequence of $4$-tuples of integers $\{(\sigma_n,\,\tau_n,\,\alpha_n,\,\beta_n)\}_{n=1,\ldots, q+1}$ defined recursively as follows: \par\noindent $(\sigma_1,\tau_1,\alpha_1,\beta_1)=(0,1,1,0),$
\begin{equation}\label{abcd}
(\sigma_{n+1},\,\tau_{n+1},\,\alpha_{n+1},\,\beta_{n+1})=\left\{
\begin{array}{lcr}
(\sigma_n-\alpha_n,\,\tau_n-\beta_n,\,\alpha_n,\,\beta_n) &\ \mbox{if} & b_n-c_n\geq c_n\\
(\alpha_n,\,\beta_n,\,\sigma_n-\alpha_n,\,\tau_n-\beta_n) &\ \mbox{if} & b_n-c_n< c_n,
\end{array}\right.
\end{equation}
for $1\leq n\leq q.$
\par For the case $u=3,\,d=10,$ we have that the sequence $\{(\sigma_n,\,\tau_n,\,\alpha_n,\,\beta_n)\}_{n=1,\ldots,6}$ is equal to
$$
\{(0,1,1,0), (-1,1,1,0), (1,0, -2,1), (3,-1,-2,1), (5,-2,-2,1), (-2,1, 7,-3)\}.
$$
Recall the definition of the sequence $\{m_\ell\}_{\ell=0,\ldots, p}$ given in the introduction:  set $m_0=1,$ and for $\ell=1,\ldots, p-1,\,m_\ell=1+\sum_{j=1}^\ell q_j.$ Set also  $m_{p}=m_{p-1}+1=q+2.$
The reason we call the sequences $\{(b_n,\,c_n)\}_{n=1,\ldots, q+1}$ and $\{(\sigma_n, \tau_n, \alpha_n, \beta_n))\}_{n=1,\ldots, q+1}$ ``slow'' Euclidean  and Extended Euclidean respectively is the following:
\begin{proposition}\label{connection}
With the notation established above, we have
\begin{enumerate}
	\item [i)] For $0\leq\ell\leq p-1$ and $m_\ell \leq n< m_{\ell+1},$ $$(b_n, c_n)=(a_{\ell}-(n-m_\ell)a_{\ell+1},a_{\ell+1}).$$
In particular,  $(b_{m_\ell}, c_{m_\ell})=(a_{\ell},a_{\ell+1})$ for $\ell\leq p-1.$
\item[ii)]  For $0\leq\ell\leq p-1$ and $m_\ell\leq n< m_{\ell+1},$
$$(\sigma_n, \tau_n, \alpha_n, \beta_n)=(s_{\ell}-(n-m_\ell)s_{\ell+1}, t_{\ell}-(n-m_\ell)t_{\ell+1}, s_{\ell+1}, t_{\ell+1}).$$
In particular,$(\sigma_{m_\ell}, \tau_{m_\ell}, \alpha_{m_\ell}, \beta_{m_\ell})=(s_{\ell}, t_\ell, s_{\ell+1}, t_{\ell+1})$ for $\ell\leq p-1.$
\item[iii)] For $n=1,\ldots q+1,$
$$\sigma_nu+\tau_n(d-u)=b_n \ \mbox{and} \ \ \alpha_nu+\beta_n(d-u)=c_n.
$$
\item[iv)]  Set $s_{p+1}=d-u$ and $t_{p+1}=u.$ Then, for $0\leq\ell\leq p-2,$
$$|s_\ell|=|\sigma_{m_\ell}|\leq |\sigma_{m_\ell+1}|\leq \ldots < |\sigma_{m_\ell+q_{\ell}-1}|\leq |s_{\ell+2}|=|\sigma_{m_{\ell+2}}|,$$
and all these numbers have the same sign. Also,
$$|t_\ell|=|\tau_{m_\ell}|\leq |\tau_{m_\ell+1}|\leq \ldots \leq |\tau_{m_\ell+q_{\ell}-1}|\leq |t_{\ell+2}|=|\tau_{m_{\ell+2}}|,$$
and all these numbers have the same sign.
\item[v)] For $0\leq\ell\leq p-2$ and $m_\ell \leq n< m_{\ell+1},$
$$|\sigma_n|\leq\frac{d-u}{a_{\ell+1}}<d-u, \ \mbox{and} \ \ |\tau_n|\leq\frac{u}{a_{\ell+1}}<u.$$
\item[vi)] $|\sigma_{q+1}|\leq\frac{d-u}{a_{p-2}}<d-u, \ \mbox{and} \ \ |\tau_{q+1}|\leq\frac{u}{a_{p-2}}<u.$
\item[vii)] For $j\leq q+1,\,|\sigma_{j}|< d-u,$ and $|\tau_{j}|<u.$
\item[viii)] $|\sigma_{q+1}-\tau_{q+1}|\leq \frac{d}2\leq |\sigma_q-\tau_q|.$
\end{enumerate}
\end{proposition}

\begin{proof}
i),\, ii),\, and iii) follow easily by induction. 
 iv) can be deduced from ii) and  Lemma \ref{corr}-ii) and iii).
\par To prove v), note that  $|\sigma_n|=|s_{\ell}-(n-m_\ell)s_{\ell+1}|=|s_\ell|+(n-m_\ell)|s_{\ell+1}|$ due to  Lemma \ref{corr}-ii). This shows that,  if $\ell\leq p-3,$
\begin{equation}\label{eqref}
 |\sigma_n|\leq |s_\ell-q_{\ell+1} s_{\ell+1}|= |s_{\ell+2}|\leq\frac{d-u}{a_{\ell+1}}<d-u
\end{equation} by   Lemma \ref{corr}-iv) and the fact that $a_{\ell+1}>a_{p-2}>1.$ An analogue bound holds for $\tau_n$ by using $t_\ell$ instead of $s_\ell$ in \eqref{eqref}. 
\par If $\ell=p-2,$ we have that \eqref{eqref} holds but without the last strict inequality. On the other hand, the fact that
$\sigma_n\,u+\tau_n(d-u)=b_n,
$
with $0<|b_n|<u$ holds, shows immediately that $|\sigma_n|<d-u$ and also that $|\tau_n|<u.$
\par For vi), as 
$(\sigma_{q+1},\tau_{q+1})=(\sigma_{m_{p-1}},\tau_{m_{p-1}})=(s_{p-1},t_{p-1}),
$  the claim also holds for this pair due to Lemma \ref{corr}-iv), and the fact that $a_{p-2}>1$ (as $a_{p-1}=1$).
\par
Now we will prove vii). By v) and vi), it is enough to prove the claim for $m_{p-2}\leq j<q+1.$ But by the ascending condition given in iii),it will suffice to show that the  claim holds for $j=q.$ But it is easy to see in this case that
\begin{equation}\label{t1t}
\begin{array}{cccccc}
\sigma_q&=&\left\{\begin{array}{lcr}
\sigma_{q+1}-(d-u)&\mbox{if}&\,\sigma_{q+1}>0\\
\sigma_{q+1}+(d-u)&\mbox{if}&\,\sigma_{q+1}\leq0
\end{array}
\right. & 
\tau_q&=&\left\{\begin{array}{lcr}
\tau_{q+1}-u&\mbox{if}&\tau_{q+1}>0\\
\tau_{q+1}+u&\mbox{if}&\tau_{q+1}\leq0
\end{array}
\right.
\end{array},
\end{equation}
so the claim follows straightforwardly, as it holds for $\sigma_{q+1}$ and $\tau_{q+1},$ thanks to vi).
\par We are left to prove viii). The  identities given in \eqref{t1t} imply that
$$|\sigma_q-\tau_q|=d-|\sigma_{q+1}-\tau_{q+1}|.
$$
On the other hand, as $m_{p-2}\leq q<q+1=m_{p-1},$ and $|s_{p-1}|<|s_{p-2}|$ due to Lemma \ref{corr}-iii), we have
\begin{equation}\label{idf}
|\sigma_q|=|s_{p-2}-(q-m_{p-2})s_{p-1}|\geq|s_{p-1}|=|\sigma_{q+1}|,
\end{equation}
the first equality thanks to iii), and the first inequality due to the fact that $s_{p-1}$ and $s_{p-2}$ have different signs (see Lemma \ref{corr}-ii) ). An analogous inequality holds for $|\tau_q|$ and $|\tau_{q+1}|.$ Identity \eqref{idf} and these two inequalities,  combined with the fact that $\sigma_n$ and $\tau_n$ have opposite signs for all  $n=1,\ldots, q+1,$ complete the proof of the claim.
\end{proof}

\bigskip
\section{Gr\"obner bases on $\S^m$ and syzygies}\label{GB}
In this section we will recall definitions and properties of Gr\"obner bases of submodules  of $\S^m$ for $m\in\N$. All the known material is classical, we refer the reader to Chapter $3$ in \cite{AL94} for proofs and further references.
\par
Denote with $\{\e_1,\ldots,\e_m\}$ the canonical basis of $\S^m$. Recall that a {\em monomial} in $\S^m$ is a vector of the type $\T^{\ualpha}\X^{\ubeta}\,\e_i,\,1\leq i\leq m$, with $\T^{\ualpha}\X^{\ubeta}$ being a monomial in $S$. A {\em term order} on the monomials in $\S^m$ is a total order $\prec$ on these monomials satisfying:
\begin{enumerate}
\item $\bU\prec \T^{\ualpha}\X^{\ubeta}\,\bU$ for every monomial $\bU\in \S^m$ and $\T^{\ualpha}\X^{\ubeta}\neq1,$ and
\item if $\bU,\,\bV$ are monomials in $\S^m$ with $\bU\prec\bV,$ then $\T^{\ualpha}\X^{\ubeta}\bU\prec \T^{\ualpha}\X^{\ubeta}\,\bV$ for every monomial $\T^{\ualpha}\X^{\ubeta}\in \S.$
\end{enumerate}
With these definitions, for an element $\bff\in \S^m$,  one defines the {\em leading monomial}, the {\em leading coefficient} and the {\em leading term} of $\bff$ in the usual way, and denotes them with
$\mbox{lm}(\bff),\,\mbox{lc}(\bff)$ and $\mbox{lt}(\bff)$ respectively.
\par
Given a submodule $M\subset\S^m,$ a set $G=\{\bg_1,\ldots,\bg_t\}\subset M$ is called a {\em Gr\"obner basis} for $M$ with respect to $\prec$ if and only if for any $\bff\in M\setminus\{\bf0\},$ there exists $i\in\{1,\ldots, t\}$ such that $\mbox{lt}(\bg_i)$ divides $\mbox{lm}(\bff).$
\par
\begin{definition}
A Gr\"obner basis $G=\{\bg_1,\,\ldots,\, \bg_t\}$ of $M$ is called {\em minimal} if for all $i,\,\mbox{lc}(\bg_i)=1$ and for $i\neq j,\,\mbox{lm}(\bg_i)$ does not divide $\mbox{lm}(\bg_j).$
\par A minimal Gr\"obner basis $G=\{\bg_1,\,\ldots,\,\bg_t\}$ of $M$ is said to be {\em reduced} if, for all $i,$ no nonzero term in $\bg_i$ is divisible by any $\mbox{lm}(\bg_j)$ for any $j\neq i.$
\end{definition}

For a given monomial order $\prec$ on $\S^m$ and any submodule $M\subset\S^m$, the {\em initial submodule} of $M$, which we denote with $\mbox{lt}(M)$, is the submodule of $\S^m$ generated by
$\{\mbox{lt}(\bff),\,\bff\in M\}.$

\begin{theorem}
For a given term order $\prec,$ every nonzero submodule $M\subset\S^m$ has a unique reduced Gr\"obner basis $G$.
If $G$ is a Gr\"obner basis of $M$, then $\langle G\rangle =M,$ and
$\mbox{\rm lt}(M)=\langle \mbox{\rm lt}(\bg),\,\bg\in G\rangle.$
\end{theorem}

\begin{proof}
See \cite[Theorems 3.5.14 \& 3.5.22,\,\& Corollary 3.5.15]{AL94}.
\end{proof}
\smallskip
We now turn our attention to modules of syzygies of submodules of $\S^m.$  Let $\T^{\ualpha}\X^{\ubeta}\e_i,\,\T^{\ualpha'}\X^{\ubeta'}\e_j$ be two monomials in $\S^m, $ the {\em least common multiple} of these two monomials (denoted $\mbox{lcm}\big(\T^{\ualpha}\X^{\ubeta}\e_i, \T^{\ualpha'}\X^{\ubeta'}\e_j\big)$) is equal to either ${\bf0}$ if $i\neq j,$ or $\T^{\max\{\ualpha,\,\ualpha'\}}\X^{\max\{\ubeta,\,\ubeta'\}}\e_i$ otherwise.
\par
Let $\bff\neq0\neq\bg\in\S^m,$ the vector
\begin{equation}\label{spol}
S(\bff,\bg)=\frac{\mbox{lcm}\big(\mbox{lm}(\bff), \mbox{lm}(\bg)\big)}{\mbox{lt}(\bff)}\,\bff-
\frac{\mbox{lcm}\big(\mbox{lm}(\bff), \mbox{lm}(\bg)\big)}{\mbox{lt}(\bg)}\,\bg\in\S^m
\end{equation}
is called the {\em S-polynomial} of $\bff,\,\bg.$ Note that the S-polynomial is actually a vector of polynomials.
\begin{theorem}
Let $G=\{\bg_1,\ldots,\bg_t\}$ be a set of non-zero vectors in $\S^m$, and $\prec$ a monomial order in $\S^m.$ Then $G$ is a Gr\"obner basis for the submodule $M=\langle \bg_1,\ldots,\bg_t\rangle\subset\S^m$ if and only if for all $i\neq j,$ we can write
\begin{equation}\label{sui}
S(\bg_i, \bg_j)=\sum_{\nu=1}^tF_{ij\nu}(\T,\X)\bg_\nu,
\end{equation}
with $F_{ij\nu}(\T,\X)\in S,$ such that
\begin{equation}\label{gcond}
\max_{1\leq\nu\leq t}\{\mbox{\rm lm}\big(F_{ij\nu}(\T,\X)\mbox{\rm lm}(\bg_\nu)\big)\}=\mbox{\rm lm}\left(S(\bg_i, \bg_j)\right).
\end{equation}
\end{theorem}
\begin{proof}
\cite[Theorem 3.5.19]{AL94}
\end{proof}
\smallskip
For a sequence $\bff_1,\ldots,\bff_s\in\S^m$, the {\em syzygy} module of this sequence is the submodule of $\S^s$ defined as
\begin{equation}\label{sis}
\mbox{\rm syz}(\bff_1,\ldots,\bff_s)=\{(\bh_1,\ldots,\bh_s)\in\S^s:\,\sum_{j=1}^s\bh_j\bff_j={\bf0}\}.
\end{equation}
Suppose that $G=\{\bg_1,\ldots,\bg_t\}$ is a Gr\"obner basis of a submodule $M$ of $\S^m$ for a term order $\prec.$ Let $\{\tilde{\e}_1,\ldots,\tilde{\e}_t\}$ be the canonical basis of $\S^t,$ and write $S(\bg_i,\bg_j)$ with $1\leq i<j\leq t$ as in \eqref{sui}. Set $\bX_{i,j}:=\mbox{lcm}\big(\mbox{lm}(\bg_i), \mbox{lm}(\bg_j)\big)$ and
\begin{equation}\label{bs}
\bs_{i,j}=\frac{\bX_{i,j}}{\mbox{lt}(\bg_i)}\,\tilde{\e}_i-
\frac{\bX_{i,j}}{\mbox{lt}(\bg_j)}\,\tilde{\e}_j-\sum_{\nu=1}^tF_{ij\nu}(\T,\X)\tilde{\e_\nu}\in\S^t.
\end{equation}
By \eqref{sis}, we easily see that $\bs_{i,j}\in\mbox{syz}(\bg_1,\ldots, \bg_t).$\par
Let ${\cF}=\{\bff_1,\ldots,\bff_t\}$ be a sequence  non-zero vectors in $\S^m,$ and $\prec$ a term order in $\S^m$. We define an order $\prec_\cF$ on the monomials of $\S^t$ as follows
\begin{equation}\label{prorder}
\T^{\ualpha}\X^{\ubeta}\tilde{\e}_i\prec_\cF\T^{\ualpha'}\X^{\ubeta'}\tilde{\e}_j\iff
\left\{\begin{array}{llr}
\mbox{lm}(\T^{\ualpha}\X^{\ubeta}\bff_i)\prec\mbox{lm}(\T^{\ualpha'}\X^{\ubeta'}\bff_j)&\,\mbox{or}&\\
\mbox{lm}(\T^{\ualpha}\X^{\ubeta}\bff_i)=\mbox{lm}(\T^{\ualpha'}\X^{\ubeta'}\bff_j)&\,\mbox{and} &\,j<i.
\end{array}
\right.
\end{equation}
We call $\prec_\cF$ the {\em order induced} by $\cF$.

\begin{theorem}\label{3713}
Let $\prec$ be a term ordering on $\S^m$. If $G=\{\bg_1,\ldots,\,\bg_t\}$ is a Gr\"obner basis of a submodule $M\subset\S^m,$ then $\{\bs_{i,j}\}_{1\leq i<j\leq t}$ is a Gr\"obner basis of $\mbox{\rm syz}(\bg_1,\ldots,\bg_t)\subset\S^t$ with respect to $\prec_G.$ Moreover,
if $\bs_{i,j}\neq{\bf0},$ then
\begin{equation}\label{liding}
\mbox{\rm lm}(\bs_{i,j})=\frac{\bX_{i,j}}{\mbox{\rm lm}(\bg_i)}\,\tilde{\e}_i, \quad 1\leq i<j\leq t.
\end{equation}
\end{theorem}
\begin{proof}
See \cite[Lemma 3.7.9 \& Theorem 3.7.13]{AL94}.
\end{proof}
\smallskip
If $m=1,$ the reader will find the usual definitions and properties of Gr\"obner bases of ideals in a ring of polynomials with coefficients in a field. 
\bigskip
\section{Minimal generators of $\mbox{ker}(\Phi_0)$}\label{minimal}
In this section we will prove Theorem \ref{mgenerators}, i.e. we will show that the family of $q+2$ polynomials $\{F_{n,b_n}(\T,\X)\}_{n=1,\ldots, q+2}$ defined in \eqref{fgen} is both a reduced Gr\"obner basis of $\mbox{ker}(\Phi_0)\subset\S,$ and a minimal set of generators of this ideal.

\begin{proposition}\label{est}
For any monomial order $\prec$ in $\S$, a reduced Gr\"obner basis of $\mbox{\rm ker}(\Phi_0)$ consists of binomials.
\end{proposition}

\begin{proof}
See \cite[Proposition 1.1 and Corollary 1.9]{ES96}.
\end{proof}
\smallskip
Clearly $\mbox{ker}(\Phi_0)$ is a bihomogeneous ideal of $\S,$ so any of its irreducible homogeneous binomials have to be of the form $T_0^a\X^\ugamma-T_1^a\X^\udelta$ with $|\ugamma|=|\udelta|$ and $\gcd(\X^\ugamma,\X^\udelta)=1.$
We will get more precisions on the exponents in the following claim.
\begin{lemma}\label{easy}
If  $T_0^a\X^\ugamma-T_1^a\X^\udelta\in\mbox{\rm ker}(\Phi_0),$ with $\ugamma,\,\udelta\in\N^3$ such that $|\ugamma|=|\udelta|$ and $\gcd(\X^\ugamma,\X^\udelta)=1,$ then this binomial is one of the following
\begin{equation}\label{binomial}
\begin{array}{lcl}
T_0^aX_0^{\gamma_0}X_2^{\gamma_2}-T_1^aX_1^{\gamma_0+\gamma_2}& \mbox{with}& a=-u \gamma_0+(d-u)\gamma_2\\
T_0^aX_1^{\delta_0+\delta_2}-T_1^aX_0^{\delta_0}X_2^{\delta_2}& \mbox{with}& a=u \delta_0-(d-u)\delta_2\\
T_0^aX_2^{\delta_0+\delta_1}-T_1^aX_0^{\delta_0} X_1^{\delta_1}& \mbox{with}& a=d \delta_0+(d-u)\delta_2\\
T_0^aX_1^{\gamma_1} X_2^{\gamma_2}-T_1^aX_0^{\gamma_1+\gamma_2}& \mbox{with}& a=u \gamma_1+d\gamma_2.
\end{array}
\end{equation}
\end{lemma}
\begin{proof}
By computing explicitly 
$\Phi_0\big(T_0^a\X^\ugamma-T_1^a\X^\udelta\big)$ with \eqref{presentacion}, and equating the latter to zero, we straightforwardly obtain that the possible distributions of supports are those appearing in \eqref{binomial}.
\end{proof}
\smallskip
\begin{definition}
Let $a,b,c$ be positive integers, with $a$ being a multiple of $\gcd(b,c).$ A solution $(\gamma_0,\delta_0)\in\N^2$ of the diophantine equation
$a=b\,\gamma-c\,\delta$
is said to be {\em minimal} if $$\gamma_0+\delta_0=\min\{\gamma'+\delta':\, a=b\,\gamma'-c\,\delta',\ (\gamma',\delta')\in\N^2\}.$$
\end{definition}
\begin{remark}\label{dioff}
It is easy to verify that there is a unique minimal solution for each triple $(a,b,c)$ of nonegative numbers with $\gcd(b,c)\,|\,a$ and $b\neq c.$ Indeed, if $(\gamma_0,\delta_0)$ is a minimal solution of  $a=b\,\gamma-c\,\delta$, then as all other integer solutions of this diophantine equation are of the form
\begin{equation}\label{soliton}
\left\{\begin{array}{ccl}
\gamma'&=&\gamma_0+\kappa\,\frac{c}{\gcd(b,c)}\\
\delta'&=&\delta_0+\kappa\,\frac{b}{\gcd(b,c)}.
\end{array}
\right. \ \kappa\in\Z,
\end{equation}
We deduce straightforwardly that a minimal solution with $a\neq0$ verifies either $\gamma_0<c,$ or $\delta_0<b,$ or both conditions at the same time. Moreover, if $(\gamma',\delta')$ is a nonegative solution of the diophantine equation $a=b\gamma-c\delta$ and either $\gamma'<c$ or $\delta'<b,$ then $(\gamma',\delta')$ is the minimal solution.
\end{remark}

 We recall the definition of $\prec_l,$ the lexicographic order on monomials of $\S$ given in the introduction, with $X_2\prec_l X_1\prec_l X_0\prec_l T_1\prec_l T_0$ if $\sigma_q>0$ or $X_2\prec_l X_0\prec_l X_1\prec_l T_1\prec_l T_0$ if $\sigma_q\leq0.$

There are always two elements in $\cF_0$ which are linear in $\X.$ Indeed, an explicit computation shows that they are the following:
\begin{equation}\label{sgen}
\begin{array}{lcl}
F_{1,d-u}(\T,\X)&=&T_0^{d-u}X_2-T_1^{d-u}X_1\\
F_{m_1,u}(\T,\X)&=&T_0^u X_1-T_1^u X_0.
\end{array}
\end{equation}
Via the identification $X_i\mapsto e_i,$ the set $\{F_{1,d-u}(\T,\X),\,F_{m_1,u}(\T,\X)\}$ turns out to be a basis of $\mbox{syz}(T_0^d,\,T_0^{d-u}T_1^u,\,T_1^d)$ regarded as an $R-$submodule of  $\R^3,$ which is a free $R$-module of rank $2,$   see \cite{cox08}. One can always show that (cf.\cite[Prop.3.6]{BJ03}):
$$\mbox{ker}(\Phi_0)\cong\langle F_{1,d-u}(\T,\X),\,F_{m_1,u}(\T,\X)\rangle:\langle T_0,T_1\rangle^\infty.
$$
\begin{remark}\label{seqq}
The elements in $\cF_0$ can be computed recursively by starting with $F_{1,d-u}(\T,\X)$ and $F_{m_1,u}(\T,\X)$ -the elements of the basis of $\mbox{\rm syz}(T_0^d,\,T_0^{d-u}T_1^u,\,T_1^d)-$  and applying recursively the properties we have used in \cite[Section 2.2]{CD13b}: as follows:
\begin{itemize}
\item Write $F_{1,d-u}(\T,\X)=T_0^{d-2u}{\bf T_0^u}X_2-T_1^{d-2u}{\bf T_1^u}X_1,$ and set
$$F_{i,d-2u}(\T,\X)=T_0^{d-2u}{\bf X_0}X_2-T_1^{d-2u}{\bf X_1}X_1,$$ with
$i=\left\{\begin{array}{cl}
2&\,\mbox{if}\ d-2u<u,\\
m_2&\, \mbox{otherwise.}
\end{array}
\right.
$
\item In general, recalling that $\ell(n)$ is such that $m_{\ell(n)-1}\leq n<m_{\ell(n)},$ we write
\begin{equation}\label{ee1}
\begin{array}{lcl}
F_{n,b_n}(\T,\X)&=&T_0^{b_n-b_{m_{\ell(n)}}}{\bf T_0^{b_{m_{\ell(n)}}}}\X^{\ualpha_n}-T_1^{b_n-b_{m_{\ell(n)}}}{\bf T_1^{b_{m_{\ell(n)}}}}\X^{\ubeta_n}\\
F_{m_{\ell(n)},b_{m_{\ell(n)}}}(\T,\X)&=&T_0^{b_{m_{\ell(n)}}}\X^{\ualpha'_n}-T_1^{b_{m_{\ell(n)}}}\X^{\ubeta'_n},
\end{array}
\end{equation}
and set
\begin{equation}\label{ee2}
F_{i,b_n-b_{m_{\ell(n)}}}:=T_0^{b_n-b_{m_{\ell(n)}}}\X^{\ualpha_n+\ubeta'_n}-T_1^{b_n-b_{m_{\ell(n)}}}\X^{\ubeta_n+\ualpha'_n},
\end{equation}
with \begin{equation}\label{ee3}
i=\left\{\begin{array}{cl}
n+1&\,\mbox{if}\ b_n-b_{m_{\ell(n)}}<b_{m_{\ell(n)}},\\
m_{\ell(n)+1}&\, \mbox{otherwise.}
\end{array}
\right.
\end{equation}
\end{itemize}
For instance, in the case $(d,u)=(10,3),$ the process is as follows:

$$\begin{array}{cclcl}
F_{3,3}(\T,\X)&=&T_0^3X_1-T_1^3X_0,&&\\
F_{1,7}(\T,\X)&=&T_0^7X_2-T_1^7X_1&=&T_0^4\cdot T_0^3X_2-T_1^4\cdot T_1^3X_1.
\end{array}$$
We set then
$$F_{2,4}(\T,\X)=T_0^4X_0X_2-T_1^4X_1^2=T_0\cdot T_0^3X_0X_2-T_1\cdot T_1^3X_1^2,$$
and a fortiori 
$$F_{6,1}(\T,\X)=T_0X_0^2X_2-T_1X_1^3.
$$
From here, everything goes straightforwardly:
$$\begin{array}{cclcl}
F_{3,3}(\T,\X)&=& T_0^3X_1- T_1^3X_0&=&T_0^2\cdot T_0X_1-T_1^2\cdot T_1X_0\\
F_{4,2}(\T,\X)&=&T_0^2 X_1^4-T_1^2X_0^3X_2&=&T_0\cdot T_0X_1^4-T_1\cdot T_1X_0^3X_2\\
F_{5,1}(\T,\X)&=&T_0X_1^7-T_1X_0^5X_2^2&&\\
F_{7,0}(\T,\X)&=&X_0^7X_2^3-X_1^7.&&
\end{array}
$$
\end{remark}

\smallskip
Via this algorithm, one gets a stronger inequality in Proposition \ref{connection}-iv), which will be useful in the sequel.
\begin{proposition}\label{stronzo}
If $n<k,\,k\neq m_{\ell(n)},$ then $|\sigma_n-\tau_n|<|\sigma_k-\tau_k|.$
\end{proposition}
\begin{proof}
Note that $|\sigma_n-\tau_n|$ is degree on the $\X$-variables of $F_{n,b_n}(\T,\X).$ By an easy induction, using \eqref{ee1},\,\eqref{ee2}, and \eqref{ee3}, one can show that always we have
$$|\sigma_n-\tau_n|>0,\ \ \forall n=1,\ldots, q+2.
$$ 
Let us fix $n$, and denote with $k$ the first integer satisfying $k>n$ and $\ell(k)\equiv\ell(n)\,\mod\,2.$ Then, it is easy to see that
$k\in\{n+1,m_{\ell(n)+1}\}.$ In both cases, due to \eqref{ee2} and \eqref{ee3}, we have that
$$\begin{array}{lclcl}
|\sigma_k-\tau_k|&=&\deg_\X(F_{k,b_k}(\T,\X))&=&\deg_\X\big(F_{n,b_n}(\X,\X)\big)+\deg_\X\big(F_{m_{\ell(n)},b_{m_{\ell(n)}}}(\T,\X)\big)
\\&>&\deg_\X\big(F_{n,b_n}(\X,\X)\big)&=&|\sigma_n- \tau_n|.
\end{array}
$$
This proves the claim for this value of $k$, and by induction one can straightforwardly show that it holds for any $k_0>n$ such that $\ell(k_0)\equiv\ell(n)\,\mbox{mod}\,2.$
\par Let $k$ be now the first integer satisfying $n<k,$ with $k\neq m_{\ell(n)}$ having the property that $\ell(n)$ and $\ell(k)$ having different parity. Computing explicitly, we get that
$k\in\{m_{\ell(n)}+1,\,m_{\ell(n)+2}\}.$
In the first case, thanks to \eqref{ee2} and \eqref{ee3}, we get that
\begin{equation}\label{sultan}
\begin{array}{lcl}
|\sigma_k-\tau_k|&=&\deg_\X(F_{k,b_k}(\T,\X))\\
&=&\deg_\X\big(F_{m_{\ell(n)},b_{m_{\ell(n)}}}(\T,\X)\big)+\deg_\X\big(F_{m_{\ell(n)+1},b_{m_{\ell(n)+1}}}(\T,\X)\big)
\\&>&\deg_X\big(F_{m_{\ell(n)+1},b_{m_{\ell(n)+1}}}(\T,\X)\big)\\
&=&|\sigma_{m_{\ell(n)+1}}-\tau_{m_{\ell(n)+1}}|>|\sigma_n- \tau_n|,
\end{array}
\end{equation}
the last inequality due to the fact that $\ell\big(m_{\ell(n)+1}\big)=\ell(n)+2\equiv\ell(n),$ combined with the case shown in the first part of this proof.
\par If $k= m_{\ell(n)+2},$ then we have that, using \eqref{ee2} and \eqref{ee3}, that there exists $j<m_{\ell(n)+2}$ such that
$$\begin{array}{lcl}
|\sigma_k-\tau_k|&=&\deg_\X(F_{k,b_k}(\T,\X))\\
&=&\deg_\X\big(F_{j,b_j}(\T,\X)\big)+\deg_X\big(F_{m_{\ell(n)+1},b_{m_{\ell(n)+1}}}(\T,\X)\big)
\\&>&\deg_X\big(F_{m_{\ell(n)+1},b_{m_{\ell(n)+1}}}(\T,\X)\big)\\
&=&|\sigma_{m_{\ell(n)+1}}-\tau_{m_{\ell(n)+1}}|>|\sigma_n- \tau_n|,
\end{array}
$$
where the last inequality holds for the same reasons as in \eqref{sultan}. This completes the proof for the first value of $k>n$ such that 
$k\neq m_{\ell(n)},$ having $\ell(k)$ and $\ell(n)$ different parities. For larger values of $k_0$ satisfying that $\ell(k_0)$ and $\ell(n)$ are not equal modulo $2$, by using again the first part of the proof (as now we have $\ell(k)\equiv\ell(k_0)\,\mbox{mod}\,2$), we get 
$$|\sigma_{k_0}-\tau_{k_0}|>|\sigma_k-\tau_k|,
$$
and from here the claim follows straightforwardly.
\end{proof}
\smallskip
\begin{proposition}\label{gb}
Consider the monomial order $\prec_l.$ Let $\mathcal{S}$ be the set made by $\mbox{\rm lm}\big(F_{q+2,0}(\T,\X)\big),$ and also by those monomials of the form
\begin{itemize}
\item  $T_0^aX_0^{\gamma_0}X_2^{\delta_0}$ with $(\gamma_0,\beta_0)$ being the minimal solution of  $a=-u\gamma+(d-u)\delta$ for $1\leq a\leq d-u,$
\par or
\item  $T_0^aX_1^{\gamma_0+\delta_0}$ with $(\gamma_0, \delta_0)$ being the minimal solution of $ a=u\gamma-(d-u)\delta,$ with $1\leq a\leq u.$ 
\end{itemize}
Then $\mathcal{S}$ generates   $\mbox{\rm lt}\big(\mbox{\rm ker}(\Phi_0)\big).$ 
\end{proposition}

\begin{proof}
Let $\cS^*$ be the set of monomials containig  $\mbox{\rm lm}\big(F_{q+2,0}(\T,\X)\big),$ and also $T_0^aX_0^{\gamma}X_2^{\delta}$ with $(\gamma,\delta)\in\N^2$ satisfying $a=-u\gamma+(d-u)\delta,$ and also by those monomials
of the form $T_0^aX_1^{\gamma+\delta}$ with $(\gamma,\delta)\in\N^2$ such that $ a=u\gamma-(d-u)\delta.$ We claim that $\cS^*$ generates $\mbox{lt}\big(\mbox{ker}(\Phi_0)\big).$
\par
Indeed, by Proposition \ref{est}, a reduced Gr\"obner basis of $\mbox{ker}(\Phi_0)$ consists of binomials. As this ideal is prime and bihomogeneous, the elements in its reduced Gr\"obner basis must be bihomogeneous and irreducible, so they belong to the list given in Lemma \ref{easy}.
Noting that $T_1\prec_l T_0,$ the fact that $\cS^*$ generates  $\mbox{lt}(\mbox{ker}(\Phi_0))$ will follow straightforwardly if we show that we can generate the initial ideal only with leading terms of binomials coming from the first two rows of \eqref{binomial}. To do this, we observe first that
the two binomials listed in \eqref{sgen} appear in the first two rows of \eqref{binomial}, which implies that $T_0^{d-u}X_2$ and $T_0^{u}X_1$ are elements of $\cS^*$. The fact that all the leading terms of binomials appearing in the last two rows of \eqref{binomial} can be ignored follows directly from these observations, due to the fact that
\begin{itemize}
\item $T_0^aX_2^{\gamma+\delta}$ with $a=d\gamma+(d-u)\delta$ is always multiple of $T_0^{d-u}X_2,$
\item $T_0^aX_1^\gamma X_2^\delta$ with $a=u \gamma+d\delta$ is either a multiple of $T_0^u X_1$ (if $\gamma>0$), or of   $T_0^{d-u} X_2$ (if $\gamma=0$).
\end{itemize}
These observations also imply that we can reduce the values of $a$ to the set $\{0,1,\ldots, d-u\}$ (resp. $\{1, 2,\ldots, u\})$) in the first (resp. second) row of \eqref{binomial} to generate all the monomials in $\cS^*.$
From here, it is very easy to show that every monomial in $\cS^*$ is a multiple of one in $\cS$, thanks to \eqref{soliton}. This concludes the proof of the claim.
\end{proof}
\smallskip
Recall now the family $\cF_0=\{F_{n, b_n}(\T,\X)\}_{n=1,\ldots, q+2}$  introduced in Section \ref{setup}, with  $F_{n,b_n}(\T,\X)$ defined in \eqref{fgen}, and set
\begin{equation}\label{sf0}
\cS_{\cF_0}=\{\mbox{lt}\big(F_{n,b_n}(\T,\X)\big),\,n=1,\ldots, q+2\}.
\end{equation}

\begin{lemma}\label{seethe}
$\cS_{\cF_0}\subset \cS.$
\end{lemma}
\begin{proof}
By   Proposition \ref{connection} v)-vi)-vii), and Proposition \ref{stronzo}, we deduce from the identity $\sigma_nu+\tau_n(d-u)=b_n$ that $(-\sigma_n,\tau_n)$ is the minimal solution of the diophantine equation $b_n=-u\gamma+(d-u)\delta$ if $\sigma_n\leq0,$ and $(\sigma_n,-\tau_n)$ is the minimal solution of the equation $b_n=u\gamma-(d-u)\delta$ if $\sigma_n>0.$  This fact, combined with the definition of $F_{n,b_n}(\T,\X)$ given in \eqref{fgen}, shows that $\cS_{\cF_0}\subset \cS,$ which concludes the proof.
\end{proof}
\smallskip
The following claim will help us to show that $\cF_0$ is a Gr\"obner basis of $\mbox{ker}(\Phi_0).$
\begin{proposition}\label{xxovni}
For $a=1,\ldots, d-u,$ let $(\gamma_a,\delta_a)$ be the minimal solution of $a=-u\gamma+(d-u)\delta,$  and $n\in\{1,\ldots, q+1\}$ the minimum such that $b_n\leq a$ and $\sigma_n\leq0.$ Then, $-\sigma_n\leq\gamma_a$ and $\tau_n\leq\delta_a.$
Analogously, for $a=1,\ldots,u,$ if  $(\gamma_a, \delta_a)$ is the minimal solution of $ a=u\gamma-(d-u)\delta,$ and  $n\in\{1,\ldots, q+1\}$ is the minimum such that $b_n\leq a$ and $\sigma_n>0,$ then, $\sigma_n\leq\gamma_a$ and $-\tau_n\leq\delta_a.$
\end{proposition}
\begin{proof}
Consider first the case  $(\gamma_a,\delta_a)$ being the minimal solution of 
$$a=-u\gamma+(d-u)\delta, \ \mbox{with}\, a\in\{1,\ldots, d-u\}.$$ The proof will be done by induction on $a.$ The case $a=1$ follows straightforwardly, due to the fact that in this case we will have $n\in\{q,q+1\},$ and $b_n=a=1,$ so the equality actually holds as we already know that
$(-\sigma_n,\tau_n)$ is the minimal solution of the diophantine equation $b_n=-u\gamma+(d-u)\delta.$
\par
Suppose now $a>1,$ and let $\ell\in\N$ be such that $m_{\ell-1}\leq n<m_\ell.$ Note that this implies
$0<b_{m_\ell}<b_n\leq a,$ and hence $a-b_{m_\ell}>0.$ Let $k$ be such that $b_k=b_n-b_{m_\ell}.$ By the definition of the SERS, note that we actually have $k\in\{n+1,\,m_{\ell+1}\}.$
We claim that $k$ is actually the minimum index such that $b_k\leq a-b_{m_\ell}$ and $\sigma_k\leq0.$ Indeed, the first inequality holds straightforwardly, and the second one is a direct consequence of  Proposition \ref{connection}-iv)   (recall that $\sigma_n\leq0$ by hypothesis). So, we only have to check that $k$ is actually the minimum.
\par Suppose first $k=n+1,$ and it is not the minimum. As the sequence $\{b_j\}_j$ is decreasing,  we would then have
$b_n\leq a-b_{m_\ell},
$ which would imply
\begin{equation}\label{aab}b_n+b_{m_\ell}\leq a.\end{equation} Note that $b_n+b_{m_\ell}=b_{n-1}$ if $n> m_{\ell-1},$ or $b_n+b_{m_\ell}=b_{j_0}$ with $j_0<m_{\ell-1}=n$ otherwise. In both cases, we would also have $\sigma_{n-1}\leq0$ or $\sigma_{j_0}\leq0,$ so  \eqref{aab} contradicts the choice of $n$ as the minimum index such that $b_n\leq a$ with $\sigma_n\leq0.$
\par
The case  $k=m_{\ell+1}$ can be treated analogously, by noting that in this case, we have
$$b_{m_{\ell+1}}=b_n-b_{m_\ell}\leq a-b_{m_\ell},$$
which implies that $n=m_\ell-1.$ Moreover, thanks again to  Proposition \ref{connection}-iv), for all $j$ such that $m_\ell\leq j<m_{\ell+1}$ we will have $\sigma_j>0.$ This implies that if $k=m_{\ell+1}$ is not the minimum satisfying the conditions of the hypothesis with $a-b_{m_\ell},$ then the next ``available'' index in the sequence will be $n,$ i.e.
we will actually have $b_n\leq a-b_{m_\ell}.$ This is the case we have just discarded above.
\par We continue with the proof of the claim:  by the definition of the SERS given in \eqref{abcd}, we actually have
$$\begin{array}{ccc}
\sigma_k&=&\sigma_n-\sigma_{m_\ell}\\
\tau_k&=&\tau_n-\tau_{m_\ell},
\end{array}
$$
so we get
$$b_k=-u(-\sigma_n+\sigma_{m_\ell})+(d-u)(\tau_n-\tau_{m_\ell}).
$$
Let  $(\gamma_{a*},\delta_{a^*})$ be the minimal solution of $a-b_{m_\ell}=-u\gamma+(d-u)\delta.$ By the inductive hypothesis we have that, for the $k$ minimum described above, $-\sigma_k\leq\gamma_{a^*}$ and $\tau_k\leq\delta_{a^*}.$
In addition, we have
$$a-b_{m_\ell}=-u\big(\gamma_a+\sigma_{m_\ell}\big)+(d-u)\big(\delta_a-\tau_{m_\ell}\big).$$
Note that $\big(\gamma_a+\sigma_{m_\ell},\delta_a-\tau_{m_\ell}\big)$ is actually a positive solution of the same diophantine equation, so we deduce
$$\begin{array}{rclcccl}
-\sigma_k&=&-\sigma_n+\sigma_{m_\ell}&\leq&\gamma_{a^*}&\leq&\gamma_a+\sigma_{m_\ell}\\
\tau_k&=&\tau_n-\tau_{m_\ell}&\leq&\delta_{a^*}&\leq&\delta_a-\tau_{m_\ell},
\end{array}
$$
which implies $-\sigma_n\leq\gamma_a,$ and $\tau_n\leq\delta_a$ as claimed.
\par The case $ a=u\gamma-(d-u)\delta$ with  $a\in\{1,\ldots, u\}$, follows mutatis mutandis the proof above. We leave the details as an exercise for the reader.
\end{proof}
\smallskip
Recall that
$\mbox{bideg}\big(F_{n,b_n}(\T,\X)\big)=\left(\deg_\T\big(F_{n,b_n}(\T,\X)\big),\,\deg_\X\big(F_{n,b_n}(\T,\X)\big)\right),$ and set $(a,b)\leq (c,d)$ if and only if $a\leq c$ and $b\leq d$. Also, recall that we denote with $\ell(n)$ the unique $\ell$ such that $m_{\ell-1} \leq n < m_\ell.$

\begin{proposition}\label{keyy}
For $j=0,\ldots, q+2,$
$$
\mbox{\rm bideg}\big(F_{k,b_k}(\T,\X)\big)\leq\mbox{\rm bideg}\big(F_{j,b_j}(\T,\X)\big)\iff k\in\{m_{\ell(j)},\,j\}.
$$
\end{proposition}

\begin{proof}
As the sequence $\{b_n\}_n,$ which keeps track of the $\T$-degrees of $\{F_{n,b_n}(\T,\X)\}$ is decreasing, we must have already $j\leq k.$
But if $k>j$ and $k\neq m_{\ell(j)},$ then thanks to Proposition \ref{stronzo}, we know that
$$\deg_\X\big(F_{k,b_k}(\T,\X)\big)=|\sigma_k-\tau_k|>|\sigma_j-\tau_j|=\deg_\X\big(F_{j,b_j}(\T,\X)\big),$$
so the claim cannot hold for this value of $k.$ Reciprocally, if $k\in\{m_{\ell(j)},\,j\},$ the statement follows straightforwardly.
\end{proof}
\smallskip
With all these auxiliary results, we can prove the first of the main theorems in the introduction.
\begin{proof}[Proof of Theorem \ref{mgenerators}]
Let $\cS$ be the set defined in the statement of Proposition \ref{gb}. By that claim, we know already that $\cS$ is a Gr\"obner basis of $\mbox{ker}(\Phi_0)$ with respect to $\prec_l.$ The fact that
$\mbox{lm}\big(\cS_{\cF_0}\big)=\mbox{lm}\big(\cS\big)$, with $\cS_{\cF_0}$ being defined in \eqref{sf0} follows straightforwardly from Lemma \ref{seethe} and Proposition \ref{xxovni}.
Indeed, Proposition \ref{xxovni} implies that for any $a$ such that $T_0^aX_0^{\gamma_0}X_2^{\delta_0}$ (resp. $T_0^aX_1^{\gamma_0+\delta_0})\,\in\cS$, there exists
$T_0^{b_n}X_0^{-\sigma_n}X_2^{\tau_n}$ (resp. $T_0^{b_n}X_1^{\sigma_n-\tau_n})\,\in\cS_{\cF_0}$ dividing this monomial.
\par This then implies that $\cS_{\cF_0}$ is a Gr\"obner basis of $\mbox{ker}(\Phi_0).$ To see that this basis is reduced, first note  that we have straightforwardly from the definition of these elements given in \eqref{fgen}, and Proposition \ref{keyy}, that the leading terms of each $F_{n,b_n}(\T,\X)$ are not divisible by $\mbox{lt}\big(F_{j,b_j}(\T,\X)\big),\,j\neq n$. This shows that the basis is {\em minimal}. The fact that it is reduced follows also immediatly by noting that all the elements of $\cF_0$ are monic and that the monomials which are not the leading term of any of these binomials cannot be divisible by any of the $\mbox{lt}(F_{j,b_j}(\T,\X)),\,j=1,\ldots q+2.$ This is because  the other monomial appearing in $F_{j,b_j}(\T,\X))$ which is not its leading term, is neither  a multiple of $T_0$ nor of $\mbox{lt}\big(F_{q+2,0}(\T,\X)\big),$ for $j=0,\ldots, q+2.$
\par It remains to show that $\cF_0$ is a minimal set of generators of $\mbox{ker}(\Phi_0).$  Suppose now that one of the $F_{n,b_n}(\T,\X)$ can actually be expressed as a polynomial combination of the other elements in $\cF_0.$ By bihomogeneity, we then must have an expression as follows:
\begin{equation}\label{abovve}
F_{n,b_n}(\T,\X)=\sum G_j(\T,\X)F_{j,b_j}(\T,\X),
\end{equation}
the sum being over those $j\neq n$ such that $\mbox{bideg}\big(F_{j,b_j}(\T,\X)\big)\leq\mbox{bideg}\big(F_{n,b_n}(\T,\X)\big).$ By Proposition \ref{keyy}, there is only possibility for such $j,$ which is $j=m_\ell$ (or none if $n=m_\ell$). But then, \eqref{abovve} would imply that
$F_{n,b_n}(\T,\X)$ is a multiple of $F_{m_\ell, b_{m_\ell}}(\T,\X)$, which contradicts the fact that it is an irreducible element. This concludes with the proof.
\end{proof}

\bigskip
\section{Minimal Generators of $\mbox{ker}(\Phi_1)$}\label{minimal1}
In this section we will work with the module of syzygies of $\cF_0$ regarded as a submodule of  $\S^{q+2}$ thanks to Theorem \ref{mgenerators}. We  associate the $n$-th element of the canonical basis $\e_n$ with the polynomial $F_{n,b_n}(\T,\X).$
As $\cF_0$ is a Gr\"obner basis of $\mbox{ker}(\Phi_0)$ with respect to $\prec_{l},$ Theorem \ref{3713} implies that the set
$\{\bs_{n,m}\}_{1\leq n<m\leq q+2}$  is  a Gr\"obner basis of
$\mbox{syz}\big(\cF_0)=\mbox{ker}(\Phi_1)$   with respect to $\prec_{l,\cF_0}$. We will work out the  syzygies presented in \eqref{2en1}, and \eqref{rrest}, and show that a subset of them are a minimal Gr\"obner basis of $\mbox{ker}(\Phi_1)$ and also a minimal set of generators of this module.
\par Recall the definitions of $\ell(n)$ and $\rho(n)$ given in the introduction, and in \eqref{rho} respectively.
\begin{lemma}\label{szg}
$$\begin{array}{ccl}
\bs_{n,\rho(n)}&=&\left\{\begin{array}{lcr}
X_0^{\sigma_{m_{\ell(n)}}}X_2^{-\tau_{m_{\ell(n)}}}\e_n-T_0^{b_{m_{\ell(n)}}}\e_{\rho(n)}-T_1^{b_{\rho(n)}}X_1^{-\sigma_n+\tau_n}\e_{m_{\ell(n)}}&\,\mbox{if}&\, \sigma_n\leq0,\\
X_1^{\tau_{m_{\ell(n)}}-\sigma_{m_{\ell(n)}}}\e_n-T_0^{b_{m_{\ell(n)}}}\e_{\rho(n)}-T_1^{b_{\rho(n)}}X_0^{\sigma_n}X_2^{-\tau_n}\e_{m_{\ell(n)}}
&\,\mbox{if}&\, \sigma_n>0,\\
\end{array}
\right.\\ \\
\bs_{n,m_{\ell(n)}}&=&\left\{\begin{array}{lcr}
X_1^{\sigma_{m_{\ell(n)}}-\tau_{m_{\ell(n)}}}\e_n-T_0^{b_{\rho(n)}}X_0^{-\sigma_n}X_2^{\tau_n}\e_{m_{\ell(n)}}-T_1^{b_{m_{\ell(n)}}}\e_{\rho(n)}
 &\,\mbox{if}&\, \sigma_n\leq0,\\
X_0^{-\sigma_{m_{\ell(n)}}}X_2^{\tau_{m_{\ell(n)}}}\e_n-T_0^{b_{\rho(n)}}X_1^{\sigma_n-\tau_n}\e_{m_{\ell(n)}}-T_1^{b_{m_{\ell(n)}}}\e_{\rho(n)}&\,\mbox{if}&\,\sigma_n>0.
\end{array}
\right.
\end{array}
$$
\end{lemma}
\begin{proof}
Assume $m_{\ell-1}\leq n<n+1<m_\ell,$ and $\sigma_n\leq0.$  All the other cases will follow analogously.  Following \eqref{spol}, we have
$$
\bS\big(F_{n,b_n}(\T,\X),\,F_{n+1,b_{n+1}}(\T,\X)\big)= X_0^{\sigma_{m_\ell}}X_2^{-\tau_{m_\ell}}\,F_{n,b_n}(\T,\X)-T_0^{b_{m_\ell}}\,F_{n+1,b_{n+1}}(\T,\X).$$ We expand this difference and  get
\begin{equation}\label{sss}
X_0^{\sigma_{m_\ell}}X_2^{-\tau_{m_\ell}}\,F_{n,b_n}(\T,\X)-T_0^{b_{m_\ell}}\,F_{n+1,b_{n+1}}(\T,\X)=T_1^{b_{n+1}}X_1^{-\sigma_n+\tau_n}\,F_{m_\ell,b_{m_\ell}}(\T,\X).
\end{equation}
 From here it is easy to deduce that
$$\bS\big(F_{n,b_n}(\T,\X),\,F_{n+1,b_{n+1}}(\T,\X)\big)=T_1^{b_{n+1}}X_1^{-\sigma_n+\tau_n}\,F_{m_\ell,b_{m_\ell}}(\T,\X),$$
and this identity satisfies \eqref{gcond}. Due to the definition of $\bs_{n,n+1}$ given in \eqref{bs}, we then have, from \eqref{sss},
$$\bs_{n,n+1}=X_0^{\sigma_{m_\ell}}X_2^{-\tau_{m_\ell}}\e_n-T_0^{b_{m_\ell}}\e_{n+1}-T_1^{b_{n+1}}X_1^{-\sigma_n+\tau_n}\e_{m_\ell},
$$
as claimed.
\end{proof}
\smallskip
The reason why we defined the lexicographic order $\prec_l$ depending on the value of $\sigma_q$ is because of the following claim.
\begin{lemma}\label{extra}
$$\begin{array}{l}
\bs_{q,q+2}=F_{q+2,0}(\T,\X)\e_{q}-F_{q,1}(\T,\X)\e_{q+2}\ \\ \\
\bs_{q+1,q+2}=\left\{\begin{array}{lcl}
X_0^{-\sigma_{q}}X_2^{\tau_{q}}\e_{q+1}-T_0\e_{q+2}-X_1^{\tau_{q+1}-\sigma_{q+1}}\e_{q} &\,\mbox{if}\,&\sigma_{q+1}\leq 0,\\
X_1^{-\sigma_q+\tau_q}\e_{q+1}-T_0\e_{q+2}-X_0^{\sigma_{q+1}}X_2^{-\tau_{q+1}}\e_{q},&\,\mbox{if}\,&\sigma_{q+1}> 0.
\end{array}
\right.
\end{array}$$
\end{lemma}

\begin{proof}
Let us start by computing $\bs_{q,q+2}.$ For $\sigma_q\leq 0,$ we get
$$\begin{array}{ccl}
\bS\big(F_{q,1}(\T,\X), F_{q+2,0}(\T,\X)\big)&=&X_1^{d}F_{q,1}(\T,\X)-T_0X_0^{-\sigma_q}X_2^{\tau_q}F_{q+2,0}(\T,\X)\\
&=&X_0^{d-u}X_2^u F_{q,1}(\T,\X)-T_1X_1^{\tau_q-\sigma_q}F_{q+2,0}(\T,\X).
\end{array}
$$
Note that the above identity satisfies \eqref{gcond}. So, by the definition of $\bs_{q,q+2}$ given in \eqref{bs}, we have
$\bs_{q,q+2}=F_{q+2,0}( \T,\X)\e_{q}-F_{q,1}(\T,\X)\e_{q+1},
$
for this case. If $\sigma_q>0,$ then
$$\begin{array}{ccl}
\bS\big(F_{q,1}(\T,\X), F_{q+2,0}(\T,\X)\big)&=&X_0^{d-u}X_2^{u}F_{q,1}(\T,\X)-T_0X_1^{\sigma_q-\tau_q}F_{q+2,0}(\T,\X)\\
&=&X_1^dF_{q,1}(\T,\X)-T_1X_0^{\sigma_q}X_2^{-\tau_q}F_{q+2,0}(\T,\X),
\end{array}
$$ which
also satisfies \eqref{gcond}, so we get again, as claimed, $$\bs_{q,q+2}=F_{q+2,0}( \T,\X)\e_{q}-F_{q,1}(\T,\X)\e_{q+1}.$$
To compute $\bs_{q+1,q+2},$ suppose first $\sigma_{q+1}\leq0.$ Note that this implies $\sigma_q>0.$ Then
$$\begin{array}{ccl}
\bS\big(F_{q+1,1}(\T,\X), F_{q+2,0}(\T,\X)\big) & = &X_0^{d-u+\sigma_{q+1}}X_2^{u-\tau_{q+1}}F_{q+1,1}(\T,\X)-T_0F_{q+2,0}(\T,\X)\\
&=&-T_1X_0^{d-u+\sigma_{q+1}}X_1^{\tau_{q+1}-\sigma_{q+1}}X_2^{u-\tau_{q+1}}+T_0X_1^d\\
&=&X_1^{\tau_{q+1}-\sigma_{q+1}}\big(T_0X_1^{d-\tau_{q+1}+\sigma_{q+1}}-T_1X_0^{d-u+\sigma_{q+1}}X_2^{u-\tau_{q+1}}\big)\\
&=&X_1^{\tau_{q+1}-\sigma_{q+1}}F_{q,1}(\T,\X).
\end{array}
$$
This equality satisfies \eqref{gcond}, so -due to \eqref{bs}- we get
$$\bs_{q+1,q+2}=X_0^{d-u+\sigma_{q+1}}X_2^{u-\tau_{q+1}}\e_{q+1}-T_0\e_{q+2}-X_1^{\tau_{q+1}-\sigma_{q+1}}\e_{q},
$$
as claimed.
\par For the case $\sigma_{q+1}>0,$  we have
$$\begin{array}{ccl}
\bS\big(F_{q+1,1}(\T,\X), F_{q+2,0}(\T,\X)\big) & = &X_1^{d-\sigma_{q+1}+\tau_{q+1}}F_{q+1,1}(\T,\X)-T_0F_{q+2,0}(\T,\X)\\
&=&X_0^{d-u+\sigma_q}X_2^{u-\tau_q}F_{q,1}(\T,\X),
\end{array}
$$
and again this identity satisfies \eqref{gcond}. So, we have, by \eqref{bs},
$$\bs_{q+1,q+2}(\T,\X)=X_1^{d-\sigma_{q+1}+\tau_{q+1}}\e_{q+1}-T_0\e_{q+2}-X_0^{d-u+\sigma_q}X_2^{u-\tau_q}\e_{q},
$$
which proves the claim by noticing that
$$\begin{array}{rcl}
d-\sigma_{q+1}+\tau_{q+1}&=&-\sigma_q+\tau_q\\
d-u+\sigma_q&=&\sigma_{q+1}\\
u-\tau_q&=&-\tau_{q+1}.
\end{array}
$$
\end{proof}
\smallskip
The following claim is a straightforward consequence of Lemmas \ref{szg} and \ref{extra}.
\begin{lemma}\label{abb}
$\bs_{q,q+2}$ belongs to the $\K[\X]-$ module generated by  $\{\bs_{q,q+1},\bs_{q+1,q+2}\}.$ Moreover,
\begin{itemize}
\item if $\sigma_q\leq 0,$ then $
\bs_{q,q+2}=X_1^{-\sigma_q+\tau_q}\bs_{q,q+1}-X_0^{-\sigma_q}X_2^{\tau_q}\bs_{q+1,q+2};$
\item if $\sigma_q>0,$ then $\bs_{q,q+2}=X_0^{-\sigma_q}X_2^{\tau_q}\bs_{q,q+1}-X_1^{\sigma_q-\tau_q}\bs_{q+1,q+2}.
$
\end{itemize}
\end{lemma}

\smallskip
For $n=1,\ldots, q,$ we define the sequences $\{k_{1n}\}_{n=1,\ldots, q},\,\{k_{2n}\}_{n=1,\ldots, q}$ as follows:
\begin{itemize}
\item if $\rho(n)=n+1,$ then $k_{1n}=\rho(n),$ and $k_{2n}=m_{\ell(n)};$
\item if $\rho(n)=m_{\ell(n)+1},$ then $k_{1n}=m_{\ell(n)},$ and $k_{2n}=\rho(n).$
\end{itemize}
Note that we always have $k_{1n}<k_{2n}.$
Let $\cF_1^*$ be the set of all $2q$ syzygies defined in Lemma \ref{szg}, and also of $\bs_{q+1,q+2}.$ Set  $\cF_1:=\cF_1^*\setminus\{\bs_{q,q+2}\}.$ Note that the definition of $\cF_1$ is consistent with the one we made in Section \ref{setup}, and also we have that the bidegrees of each of these syzygies regarded as elements of $ \displaystyle{\bigoplus_{n=1}^{q+2}}\S(-(b_n, |\sigma_n-\tau_n|))$ satisfy \eqref{25252} if we declare $\bdeg(\e_n)=(b_n, | \sigma_n-\tau_n|)$.
\smallskip

\begin{proof}[Proof of Theorem \ref{phi1}]
Theorem \ref{3713} implies  that the family $\{\bs_{n,m}\}_{1\leq n<m\leq q+2}$ is a Gr\"obner Basis of $\mbox{ker}(\Phi_1)=\mbox{syz}(\cF_0)$ with respect to $\prec_{l,\cF_0}.$ We will first show that the leading module generated by this family is the same as the one generated by $\mbox{lt}(\cF_1^*),$ which will then imply that $\cF_1^*$ is also a Gr\"obner basis of this module.
\par
To show this, we compute explicitly the leading term of $\bs_{n,m}$  for $1\leq n<m< q+2.$  By using \eqref{liding}, we get straightforwardly
\begin{equation}\label{pprimer}
\mbox{lt}\big(\bs_{n,m}\big)=\left\{\begin{array}{lcl}
X_0^{-\sigma_m+\sigma_n}X_2^{\tau_m-\tau_n}\e_n &\,\mbox{if}&\,\sigma_n,\,\sigma_m\leq0\\
X_0^{-\sigma_m}X_2^{\tau_m}\e_n &\,\mbox{if}&\,\sigma_n>0,\,\sigma_m\leq0\\
X_1^{\sigma_m-\tau_m}\e_n &\,\mbox{if}&\,\sigma_n\leq0,\,\sigma_m>0\\
X_1^{\sigma_m-\tau_m-\sigma_n+\tau_n}\e_n &\,\mbox{if}&\,\sigma_n,\,\sigma_m>0.
\end{array}
\right.
\end{equation}
We also compute, for $n<q+2,$
$$
\mbox{lt}\big(\bs_{n,q+2}\big)=\left\{\begin{array}{lcl}
X_1^{d}\e_n &\,\mbox{if}&\,\sigma_n,\,\sigma_q\leq0\\
X_1^{d-\sigma_n+\tau_n}\e_n &\,\mbox{if}&\,\sigma_n>0,\,\sigma_q\leq0\\
X_0^{d-u+\sigma_n}X_2^{u-\tau_n}\e_n &\,\mbox{if}&\,\sigma_n\leq0,\,\sigma_q>0\\
X_0^{d-u}X_2^u\e_n &\,\mbox{if}&\,\sigma_n,\,\sigma_q>0.
\end{array}
\right.
$$
Note that for a fixed $n<q,$ there will be two minimal elements in the set $\{\mbox{lt}(\bs_{n,m})\}_{n<m},$ which will be found in the group \eqref{pprimer}. Indeed, suppose $\sigma_n\leq0$. Then, the minimal elements in this set are
$$X_0^{-\sigma_{m_0}+\sigma_n}X_2^{\tau_{m_0}-\tau_n}\e_n,\quad X_1^{\sigma_{m'_0}-\tau_{m'_0}-\sigma_n+\tau_n}\e_n,
$$
with $-\sigma_{m_0}$ and $\tau_{m_0}$ being minimal among those elements with $\sigma_m\leq0,$ and $\sigma_{m'_0}-\tau_{m'_0}$ being minimal among those with $\sigma_m>0.$ By the definition of the sequences $\{\sigma_m,\,\tau_m\}$ in \eqref{abcd}, and the properties of these sequences given in Proposition \ref{connection}-iv) and Proposition \ref{stronzo}, we get straightforwardly that  the minimums are achieved at
\begin{itemize}
\item $m_0=\rho(n)$ and $m'_0=m_{\ell(n)}$  if $\rho(n)=n+1,$ 
\item $m_0=m_{\ell(n)}$ and $m'_0=\rho(n)$ if $\rho(n)=m_{\ell(n)+1}.$
\end{itemize}
In both cases, the minimums coincide with $k_{2n}$ and $k_{1n}$ respectively. The case $\sigma_n>0$ follows analogously.
Note that the previous analysis excluded $\bs_{q,q+2}$ and $\bs_{q+1,q+2},$ but these two elements belong to $\cF_1^*$, so we
then deduce that $\mbox{lt}(\cF_1^*)$ generates $\mbox{lt}\big(\mbox{\rm ker}(\Phi_1)\big),$ and hence it is a Gr\"obner basis of this module.
\par
Let us see now that $\cF_1$ is a minimal Gr\"obner basis. To do this, note that by \eqref{liding} and Lemma \ref{extra}, $\mbox{lt}(\bs_{q,q+1})$ divides $\mbox{lt}(\bs_{q,q+2}),$ so this element can be removed from the list. This shows that $\cF_1$ is also a Gr\"obner basis of
$\mbox{\rm ker}(\Phi_1),$ and the fact that it is minimal follows straightforwardly due to the fact that, for each $n=1,\ldots, q+1,$ there are at most two monomials of the form $\X^{\ualpha_n}\e_n$ and $\X^{\ualpha'_n}\e_n$ which are leading terms of elements in the family $\cF_1.$ These monomials have disjoint support, so if we remove one of them from the family, we do not generate the same initial module. So, we have shown that $\cF_1$ is a minimal Gr\"obner basis. The fact that the basis is reduced follows by analyzing the explicit forms of the elements of $\cF_1$ given in Lemmas \ref{szg} and \ref{extra}. We leave the details to the reader.
\par
To conclude, we must prove that $\cF_1$ is a minimal set of generators.
Suppose there exists $\bs_{n,k_{in}}\in\cF_1$ which can be written as a polynomial combination of the others. Suppose first $n<q.$ We then have,
$$\begin{array}{ccl}
\bs_{n,k_{in}}&=&A(\T,\X)\bs_{n,k_{jn}}+\sum_{m\neq n}\big(B_{1m}(\T,\X)\bs_{m,k_{1m}}+B_{2m}(\T,\X)\bs_{m,k_{2m}}\big)\\
&&+C(\T,\X)\bs_{q,q+1}+D(\T,\X)\bs_{q+1,q+2},
\end{array}
$$
with $A(\T,\X),\,B_{im}(\T,\X),\,C(\T,\X),\,D(\T,\X)\in\S,$ and $\{i,\,j\}=\{1,\,2\}.$ Set $\T\mapsto0$ in the identity above. Due to the explicit form of these syzygies shown in Lemma \ref{szg}, and the fact that $b_k=0$ if and only if $k=q+2,$ we will have
$$\begin{array}{ccl}
\X^{\ualpha}\e_n&=&A({\bf0},\X)\X^{\ubeta}\e_n+\sum_{m\neq n}\big(B_{1m}({\bf0},\X)\X^{\ualpha_m}+B_{2m}({\bf0},\X)\X^{\ubeta_m}\big)\e_m\\
&&+C({\bf0},\X)\e_q+D({\bf0},\X)\e_{q+1},
\end{array}
$$
for some $\ualpha,\,\ubeta,\,\ualpha_m,\,\ubeta_m$ such that $\gcd(\X^\ualpha,\,\X^\ubeta)=1.$ By comparing the coefficients of $\e_n$ in both sides of the identity above, we have
$$\X^{\ualpha}=A({\bf 0},\X)\X^{\ubeta},
$$
which is impossible as this would imply $\X^{\ubeta}$ being a divisor of $\X^{\ualpha},$ a contradiction.
\par  For $n=q$, we get
$$
\bs_{q,q+1}=\sum_{m<q}\big(B_{1m}(\T,\X)\bs_{m,k_{1m}}+B_{2m}(\T,\X)\bs_{m,k_{2m}}\big)+
D(\T,\X)\bs_{q+1,q+2},
$$
with $B_{km}(\T,\X),\,D(\T,\X)\in\S.$ By setting $\T\mapsto{\bf0},$ we have now -due to Lemmas \ref{szg} and  \ref{extra}-
$$\X^{\ualpha}\e_q+\X^{\ubeta}\e_{q+1}=\sum_{m<q}\big(B_{1m}({\bf0},\X)\X^{\ualpha_m}+B_{2m}({\bf0},\X)\X^{\ubeta_m}\big)\e_m+D({\bf0},\X)\big(\X^{\ualpha'}\e_q+\X^{\ubeta'}\e_{q+1}\big),
$$
with $\gcd\big(\X^\ualpha,\X^{\ualpha'}\big)=\gcd\big(\X^\ubeta,\X^{\ubeta'}\big)=1.$ The claim now follows by comparing the coefficients of $\e_q$ in the above identity. The case $n=q+1$ follows mutatis mutandis this case.
\end{proof}

\bigskip
\section{Minimal Generators of $\mbox{ker}(\Phi_2)$}\label{minimal2}
In this section, we will prove Theorem \ref{phi2}. First, we make explicit the nontrivial syzygies in the family $\cF_2$ defined in the introduction.
The following claim can be checked straightforwardly from the formulas given in Lemma \ref{szg}.

\begin{lemma}\label{extra2}
In $\S^{2q+1},$ with the order induced by $\prec_{l,\cF_1^*},$ we have:
\begin{equation}\label{zuper}
\begin{array}{rcl}
\bS\big(\bs_{n,\rho(n)}, \bs_{n,m_{\ell(n)}}\big)&=&F_{\rho(n),b_{\rho(n)}}(\T,\X)\e_{m_{\ell(n)}}-F_{m_{\ell(n)},b_{m_{\ell(n)}}}(\T,\X)\,\e_{\rho(n)},\\
T_1^{b_{m_{\ell(n)}}}\bs_{n,\rho(n)}-T_0^{b_{m_{\ell(n)}}}\bs_{n,m_{\ell(n)}}&=&F_{n,b_n}(\T,\X)\e_{m_{\ell(n)}}-F_{m_{\ell(n)},b_{m_{\ell(n)}}}(\T,\X)\e_n.
\end{array}
\end{equation}\end{lemma}

We consider now the module $\mbox{syz}(\cF_1)\subset\S^{2q},$ where we denote with $\{\e_{n,k_{in}}\}_{s_{n,k_{in}}\in\cF_1},$ the canonical basis of $\S^{2q}.$ We consider also $\prec_{l,\cF_1},$ the order induced by $\cF_1$, as defined in \eqref{prorder}, where we sort the pairs $\{(n,k_{in})\}$ with the lexicographic order of $\N\times\N.$

\begin{proposition}\label{aiecomae}
For $n=1,\ldots, q,$ we have
\par If $\sigma_n\leq0$ and $\rho(n)=n+1,$  then
$$\begin{array}{ccl}
\bs_{(n,\rho(n)),(n,m_{\ell(n)})} &=&
X_1^{\sigma_{m_{\ell(n)}}-\tau_{m_{\ell(n)}}}\e_{n,\rho(n)}-X_0^{\sigma_{m_{\ell(n)}}}X_2^{-\tau_{m_{\ell(n)}}}\e_{n,m_{\ell(n)}}\\ && +T_0^{b_{m_{\ell(n)}}}\e_{\rho(n),m_{\ell(n)}}-T_1^{b_{m_{\ell(n+1)}}}\e_{\rho(n),\rho(n+1)}.
\end{array}
$$
If  $\sigma_n\leq0$ and $\rho(n)=m_{\ell(n)+1},$ then
$$\begin{array}{ccl}
\bs_{(n,m_{\ell(n)}),(n,\rho(n))}&=&X_0^{\sigma_{m_{\ell(n)}}}X_2^{-\tau_{m_{\ell(n)}}}\e_{n,m_{\ell(n)}}-X_1^{\sigma_{m_{\ell(n)}}-\tau_{m_{\ell(n)}}}\e_{n,\rho(n)}\\ &&+T_0^{b_{\rho(n)}}\e_{n+1,\rho(n)}-T_1^{b_{\rho(n)}}\e_{n+1,\rho(n+1)}.
\end{array}$$
If  $\sigma_n>0$ and $\rho(n)=n+1,$ then
$$\begin{array}{ccl}
\bs_{(n,\rho(n)),(n,m_{\ell(n)})}&=&
X_0^{-\sigma_{m_{\ell(n)}}}X_2^{\tau_{m_{\ell(n)}}}\e_{n,\rho(n)}-X_1^{\tau_{m_{\ell(n)}}-\sigma_{m_{\ell(n)}}}\e_{n,m_{\ell(n)}}\\
&&+T_0^{b_{m_{\ell(n)}}}\e_{\rho(n),m_{\ell(n)}}-T_1^{b_{m_{\ell(n+1)}}}\e_{\rho(n),\rho(n+1)}.
\end{array}
$$
If  $\sigma_n>0$ and $\rho(n)=m_{\ell(n)+1},$ then
$$\begin{array}{ccl}
\bs_{(n,m_{\ell(n)}),(n,\rho(n))}&=&X_1^{\tau_{m_{\ell(n)}}-\sigma_{m_{\ell(n)}}}\e_{n,m_{\ell(n)}}-X_0^{-\sigma_{m_{\ell(n)}}}X_2^{\tau_{m_{\ell(n)}}}\e_{n,\rho(n)}\\
&&+T_0^{b_{\rho(n)}}\e_{n+1,\rho(n)}-T_1^{b_{\rho(n)}}\e_{n+1,\rho(n+1)}.
\end{array}
$$
\end{proposition}

\begin{proof}
We will prove the first case, all the others follow analogously. For a fixed $n,$ we will denote by $\ell$ the value $\ell(n).$ So, we have
$m_{\ell-1}\leq n<n+1<m_\ell,$ and $\sigma_n\leq0.$ By using definition \eqref{spol},  identity \eqref{pprimer} and Lemma \ref{extra2}, we have
\begin{equation}\label{kokk}
\begin{array}{ccl}
\bS\big(\bs_{n,n+1}, \bs_{n,m_{\ell}}\big)&=&F_{n+1,b_{n+1}}(\T,\X)\e_{m_{\ell}}-F_{m_{\ell},b_{m_{\ell}}}(\T,\X)\,\e_{n+1}\\
&=&X_1^{\sigma_{m_\ell}-\tau_{m_\ell}}\bs_{n,n+1}-X_0^{\sigma_{m_\ell}}X_2^{-\tau_{m_\ell}}\bs_{n,m_\ell}.
\end{array}
\end{equation}
From the expression above and the definition of $\prec_{\l,\cF_1}$ given in \eqref{prorder}, it is easy to  compute the leading term of $\bS\big(\bs_{n,n+1}, \bs_{n,m_{\ell}}\big)$ which turns out to be $-T_0^{b_{m_\ell}}X_1^{\sigma_{m_\ell}-\tau_{m_\ell}}\e_{n+1}.$ So, to get an expression like \eqref{sui} satisfying \eqref{gcond}, we must substract from \eqref{kokk} $-T_0^{b_{m_\ell}}\,\bs_{n+1,m_\ell},$ and get
$$\begin{array}{ccl}
\bS\big(\bs_{n,n+1}, \bs_{n,m_{\ell}}\big)&=&\big(F_{n+1,b_{n+1}}(\T,\X)\e_{m_{\ell}}-F_{m_{\ell},b_{m_{\ell}}}(\T,\X)\,\e_{n+1}+T_0^{b_{m_\ell}}\,\bs_{n+1,m_\ell}\big)-T_0^{b_{m_\ell}}\,\bs_{n+1,m_\ell}\\
&=&T_1^{b_{m_\ell}}\bs_{n+1,\rho(n+1)}-T_0^{b_{m_\ell}}\,\bs_{n+1,m_\ell},
\end{array}
$$
where the last equality follows from the second identity in \eqref{zuper} with $n$ replaced by $n+1$ (note that in this case $\ell(n+1)=\ell$). Clearly, this expression satisfies \eqref{gcond} and so, due to \eqref{bs}, we have then
$$\bs_{(n,n+1), (n,m_\ell)}=
X_1^{\sigma_{m_\ell}-\tau_{m_\ell}}\e_{n,n+1}-X_0^{\sigma_{m_\ell}}X_2^{-\tau_{m_\ell}}\e_{n,m_\ell}-T_1^{b_{m_\ell}}\e_{n+1,\rho(n+1)}+T_0^{b_{m_\ell}}\,\e_{n+1,m_\ell},
$$
which is what we wanted to prove.
\end{proof}
\smallskip
Note that all the syzygies considered in Proposition \ref{aiecomae} are either of the form
$\bs_{(n,\rho(n)),(n,\ell(n))},$ or $\bs_{(n,\ell(n))(n,\rho(n))}.$ We will denote them with $\bs_{n,\rho(n),\ell(n)}$ for short. With this notation,  \eqref{consistent2} holds thanks to Proposition \ref{aiecomae}.

\begin{proof}[Proof of Theorem \ref{phi2}]
We apply again Theorem \ref{3713} to $\cF_1$ which we know is a Gr\"obner basis for $\prec_{l,\cF_0}$ of $\mbox{syz}(\cF_0)\simeq\mbox{ker}(\Phi_1)$ by Theorem \ref{phi1}. We get then that $\{\bs_{(n,k),\,(n',k')}\}_{(n,k)\prec (n',k')}$ is a Gr\"obner basis of $\mbox{syz}(\cF_1),$ where $\prec$ denotes the lexicographic order in $\N\times\N.$
\par We will first detect which of these syzygies are not zero. From \eqref{liding}, we get that
$$\mbox{\rm lm}\big(\bs_{n,k}\big)=\T^{\ualpha}\X^\ubeta\,\e_n,
$$ for some monomial $\T^{\ualpha}\X^\ubeta.$This shows that for $n\neq n',$ we will have
$$\mbox{\rm lcm}\left(\mbox{\rm lm}\big(\bs_{n,k}\big),\mbox{\rm lm}\big(\bs_{n',k'}\big)\right)={\bf0},
$$ as these monomials do not have the same support. Hence,  $\bs_{(n,k),(n',k')}={\bf0}$ as well. This implies that the only nonzero syzygies are actually those in $\cF_2.$ To see that the basis is reduced, we just have to note that, due to \eqref{liding} again,
\begin{equation}\label{remate}
\mbox{\rm lt}\big(\bs_{n,\rho(n),\ell(n)}\big)=\T^\ualpha\X^\ubeta\e_{n,k_{n}}
\end{equation}
for a monomial $\T^\ualpha\X^\ubeta\in\S,$ and $k_n\in\{\rho(n),\,\ell(n)\}.$ This is the only element in $\cF_2$ having  a leading term being a monomial in the coordinate $\e_{n,k_{n}},$ so cannot be neither reduced or removed, which implies straightforwardly that $\cF_2$ is a minimal basis of this module. The fact that it is reduced follows easily by inspecting the explicit form of the elements of $\cF_2$ given in Proposition \ref{aiecomae} .
\par
We are left to see that the set $\cF_2$ is $\S$-linearly independent. This will be done by considering now $\mbox{syz}(\cF_2).$ We aply Theorem \ref{3713} to $\cF_2$ and know that this syzygy module is generated by all the syzygies of elements of $\cF_2$ with respect to $\prec_{l,\cF_2}.$ But due to \eqref{remate}, we see that different syzygies have leading terms in different coordinates, which implies that
$$\bS\left(\bs_{n,\rho(n),\ell(n)},\bs_{n',\rho(n'),\ell(n')}\right)={\bf0}
$$
if $n\neq n'.$ This shows that $\mbox{syz}(\cF_2)={\bf0},$ or equivalently that the family $\cF_2$ is $\S$-linearly independent.
\end{proof}

\bigskip
\section{Adjoints}\label{adj}
In this section, we will state and prove  Lemma \ref{ddim} and Theorem \ref{adjj},  from which one deduces Theorem \ref{kkk} straightforwardly. All along this section we will assume that $\K$ is an algebraically closed field of characteristic zero, and also that $u>1,$  as the case $u=1$ corresponds to the so-called {\em monoid curve}, and the study of pencils of adjoints for this family of parametrizations has been covered already in \cite[Propositions 4.3 \& 4.4]{bus09}.

\begin{lemma}\label{sing}
Let $\cC_{u,d}$ be the rational curve defined as the image of \eqref{varphi}. If $u>1,$ then $\cC_{u,d}$ has two singular points:
\begin{itemize}
\item $\bp_0:=(1:0:0)$ with multiplicity $u,$
\item $\bp_\infty:=(0:0:1)$ with multiplicity $d-u.$
\end{itemize}
\end{lemma}
\begin{proof}
Recall that $\pm F_{q+2,0}(\T,\X)=X_0^{d-u}X_2^u-X_1^d$ is an irreducible polynomial defining $\cC_{u,d}.$  To compute the singular points of this curve, we have to look at the zeroes of the partial derivatives of this polynomial. As $u>1,$ one easily gets that
$$\mbox{Sing}(\cC_{u,d})=\{(1:0:0),\,(0:0:1)\}.$$
Computing the ``affine'' Taylor expansion of $F_{q+2,0}(\T,\X)$ around these points, we easily get the multiplicities, which concludes the proof.
\end{proof}
\smallskip
\begin{remark}
Lema \ref{sing} implies that the singularities of $\cC_{u,d}$ are not ordinary, as if this were the case, then by applying the genus formula (see for instance \cite{wal50}) we would get
$$u(u-1)+(d-u)(d-u-1)=(d-1)(d-2),
$$
which is impossible unless $u+1=d,$ contradicting the fact $u<\frac{d}{2}.$
\end{remark}

The following result will be useful to compute dimensions of pencils of adjoints.
\begin{lemma}\label{diof}
For $j\leq(d-1)(u-1),$ (resp. $j\leq (d-1)(d-u-1)$) there is at most one $\ualpha=(\alpha_1,\,\alpha_2,\,\alpha_3)\in\N^3$ such that  $u\alpha_1+d\alpha_2=j$ (resp.
$d\alpha_0+(d-u)\alpha_1=j$).
\end{lemma}
\begin{proof}
Suppose
that both $\ualpha$ and $\ualpha'$ satisfy these conditions, then we must have
$$(\alpha_1-\alpha'_1,\alpha_2-\alpha'_2)=\kappa(d,-u),\ \ \kappa\in\Z.$$
If $\kappa>0,$ this implies $\alpha_1=\alpha'_1+\kappa\,d\geq d,$  but then we have $$u\alpha_1+d\alpha_2\geq u\,d>(d-1)(u-1)\geq j,$$ a contradiction. If $\kappa<0,$ then we get
$\alpha_2\geq u,$ and arrive to a contradiction as above. The other inequality follows analogously.
\end{proof}
\smallskip
\begin{proposition}\label{key}
For $\ell\in\N,$ let $E_\ell(\X)\in S$ be a homogeneous form of degree $\ell$  defining a curve $\cE_\ell\subset\P^2_\K.$ Write
 $$E_\ell(\X)=\sum_{|\ualpha|=\ell}e_\ualpha\X^{\ualpha}.$$
Then, $\cE_\ell$ is adjoint to $\cC_{u,d}$ if and only if
$e_\ualpha=0$ for all $\ualpha=(\alpha_0,\alpha_1,\alpha_2)\in\N^3$ such that either
$u\alpha_1+d\alpha_2<(d-1)(u-1)$ or $d\alpha_0+(d-u)\alpha_1<(d-1)(d-u-1)$.
\end{proposition}

\begin{proof}
We will use the characterization of adjoints given in Theorem $6.3.1$ in \cite{cas00} (see also \cite[Proposition 4.6]{bus09}) for the case of $\cC_{u,d}.$ We start by choosing the point $(0:1:0)\notin\cC_{u,d},$ and easily see that the {\em polar curve} with respect to this point is defined by $\frac{\partial F{q+2,0}(\T,\X)}{\partial X_1}=\pm dX_1^{d-1}.$ By inspecting the  parametrization $\varphi$ of $\cC_{u,d}$ given in \eqref{varphi}, we deduce straithgforwardly that for each $i\in\{0,\infty\},$ there is only one (irreducible) branch $\gamma_i(t)$ of the curve passing through $\bp_i.$ Computing these branches explicitly from $\varphi,$ we get that one representation of them can be the following:
\begin{itemize}
\item $\gamma_0(t)=(1:t^u:t^d),$
\item $\gamma_\infty(t)=(t^d:t^{d-u}:1).$
\end{itemize}
These branches are irreducible, as  $\gcd(u,d)=\gcd(d,d-u)=1.$
\par We apply now \cite[Theorem 6.3.1]{cas00} to this situation, and get that $\cE_\ell$ is adjoint to $\cC_{u,d}$ if and only if
\begin{equation}\label{cococo}
\begin{array}{lcl}
\mbox{m}_{\bp_0}\big(\gamma_0,\cE_\ell\big)&\geq& (d-1)(u-1)\\
\mbox{m}_{\bp_\infty}\big(\gamma_\infty,\cE_\ell\big)&\geq&(d-1)(d-u-1),
\end{array}\end{equation}
where $\mbox{m}_{\bp_i}(\gamma_i,\cE_\ell)$ denotes the local multiplicity of the branch $\gamma_i$ at the point $\bp_i$ with respect to $\cE_\ell.$ To compute these local multiplicities, we set $\X\mapsto\gamma_i(t)$ in $E_\ell(\X)$, and get
\begin{itemize}
\item For $\bp_0:\ E_\ell\big(\gamma_0(t)\big)=\sum_{|\ualpha|=\ell}e_\ualpha t^{u\alpha_1+d\alpha_2},$
\item for $\bp_\infty:\ E_\ell\big(\gamma_\infty(t)\big)=\sum_{|\ualpha|=\ell}e_\ualpha t^{d\alpha_0+(d-u)\alpha_1}.$
\end{itemize}
so, we must have
\begin{equation}\label{arroc}
\begin{array}{lr}
\min_j \{u\alpha_1+d\alpha_2=j,\,\sum_{|\ualpha|=j}e_\ualpha\neq0\}\geq (d-1)(u-1),& \mbox{and} \\
\min_j \{d\alpha_0+(d-u)\alpha_1=j,\,\sum_{|\ualpha|=j}e_\ualpha\neq0\}\geq (d-1)(d-u-1).&
\end{array}
\end{equation}

\par The claim follows straightforwardly from Lemma \ref{diof} combined with \eqref{arroc}.
\end{proof}
\smallskip
For rational plane curves of degree $d,$ it is well-known (see for instance \cite{SWP08}) that there are no adjoint curves of degree less than $d-2$. We can recover this result for the monomial curve from Proposition \ref{key} above.
\begin{corollary}
If $\ell<d-2,$ then $\mbox{\em Adj}_\ell(\cC_{u,d})=0.$
\end{corollary}
\begin{proof}
A nontrivial adjoint $\cE_\ell$ to $\cC_{u,d}$ will have a polynomial defining it, which we denote by $E_\ell$. Computing its Taylor expansion, we get  $E_\ell=\sum_{|\ualpha|=\ell}e_{\ell}\X^\ualpha,$ with some $e_\ualpha\neq0.$ By Proposition \ref{key}, this can happen if and only if
$$u\alpha_1+d\alpha_2\geq(d-1)(u-1)\ \mbox{and} \ d\alpha_0+(d-u)\alpha_1\geq (d-1)(d-u-1).
$$
Adding these two inequalities, we get
$$d\,\ell=d(\alpha_0+\alpha_1+\alpha_2)\geq (d-1)(d-2),
$$
which implies $\ell\geq\frac{(d-1)(d-2)}{d}.$ From here, we get straightforwardly that $\ell\geq d-2.$
\end{proof}
\smallskip
The following classic properties will also be useful in the sequel. They were already known by Sylvester (see \cite{syl84}).
\begin{lemma}\label{aixx}
Let $a,\,b$ be positive integers such that $\gcd(a,b)=1.$ Then
\begin{enumerate}
\item[i)] The number of $j\in\N$ such that the positive Diophantine equation
\begin{equation}\label{pos}
a\cdot x+b\cdot y=j,\ \ \mbox{ with}, x,\,y\in\N,
\end{equation} has no  solution is equal to $(a-1)(b-1)/2.$
\item[ii)] If $j\geq (a-1)(b-1),$ then the positive Diophantine equation \eqref{pos} has always a solution.
\end{enumerate}
\end{lemma}

\begin{lemma}\label{yek}
For $\ell\geq d-2,$ the cardinality of the set of $\ualpha=(\alpha_0,\alpha_1,\alpha_2)\in\N^3$ with $\alpha_0+\alpha_1+\alpha_2=\ell,$ such that either
$u\alpha_1+d\alpha_2<(d-1)(u-1),$ or $d\alpha_0+(d-u)\alpha_1<(d-1)(d-u-1),$ is equal to $\frac{(d-1)(d-2)}2.$
\end{lemma}

\begin{proof}
First, let us show that if $\ell\geq d-2,$ then the two conditions
$$u\alpha_1+d\alpha_2<(d-1)(u-1) \ \mbox{ and}\ \  d\alpha_0+(d-u)\alpha_1<(d-1)(d-u-1),$$ cannot happen at the same time. Indeed, if this were the case, by adding the two inequalities we would have
$d\ell< (d-1)(d-2),$ which implies $\ell<d-2,$ a contradiction.
\par
By Lemma \ref{aixx}, the number of  $j$'s such that $u\alpha_1+d\alpha_2=j$ has at least a nonegative solution (resp. $d\alpha_0+(d-u)\alpha_1=j,$ with $j<(d-1)(u-1)$ (resp. $j<(d-1)(d-u-1)$) is equal to $\frac{(d-1)(u-1)}2$ (resp. $\frac{(d-1)(d-u-1)}{2}$). By Lemma \ref{diof}, we actually have that if there is a nonegative solution of the aforementioned diophantine equation, then it is unique. This implies that the cardinality we want to compute is equal to
$$\frac{(d-1)(u-1)}2+\frac{(d-1)(d-u-1)}2=\frac{(d-1)(d-2)}2.
$$
\end{proof}
\smallskip
With all these preliminary results, we can compute the dimension of the space of pencils of adjoints.
\begin{proposition}
For $\ell\geq d-2,$
$$\dim_\K\left(\mbox{\em Adj}_\ell(\cC_{u,d})\right)= (\ell+2)(\ell+1)-(d-1)(d-2).$$
\end{proposition}

\begin{proof}
By Proposition \ref{key}, we have that  the dimension of $\mbox{Adj}_\ell(\cC_{u,d})$ is equal to twice the cardinality of those $\ualpha=(\alpha_0,\alpha_1,\alpha_2)\in\N^3$ such that $\alpha_0+\alpha_1+\alpha_2=\ell,\,u\alpha_1+d\alpha_2\geq (d-1)(u-1),$ and $d\alpha_0+(d-u)\alpha_1\geq (d-1)(d-u-1)$. This cardinality, thanks to Lemma \ref{yek}, is equal to
$${\ell+2\choose 2}-\frac{(d-1)(d-2)}2,
$$
so by multiplying by two this number, the claim holds.
\end{proof}
\smallskip
Now we turn to the computation of pieces of $\mbox{ker}(\Phi_0)_{(1,*)},$ as we want to show that the elements of higher $\X$-degree in this space cannot be generated by pencils of adjoints of $\cC_{u,d}.$ Recall that if $a<b,$ we set ${a\choose b}=0.$
\begin{lemma}\label{ddim}
$\mbox{ker}(\Phi_0)_{(1,*)}$ is a free $\K[\X]$-module, with basis $\{F_{q,1}(\T,X)\,\,F_{q+1,1}(\T,X)\}.$ For $\ell\in\N,$ we have that
$$\dim_{\K}\left(\mbox{\rm ker}(\Phi_0)_{(1,\ell)}\right)= {\ell-|\sigma_q-\tau_q|+2\choose 2}+{\ell-|\sigma_{q+1}-\tau_{q+1}|+2\choose 2}.
$$
\end{lemma}
\begin{proof}
By Theorem \ref{mgenerators}, we have that the piece of $\T$-degree one in $\mbox{ker}(\Phi_0)$ is generated by $F_{q,1}(\T,X)\,\,F_{q+1,1}(\T,X),$ and multiples of $F_{q+2,0}(\T,\X)$ of degree one in $\T$. But it is easy to see that $T_0F_{q+2,0}(\T,\X)$ and $T_1F_{q+2,0}(\T,\X)$ can be expressed as combinations of $F_{q,1}(\T,X)\,\,F_{q+1,1}(\T,X),$
so we have that these last two elements generate $\mbox{ker}(\Phi_0)_{1,*}$ as a $\K[\X]$-module. As they are part of a reduced Gr\"obner basis of a prime ideal, they must be irreducible polynomials, which implies straightforwardly  that they are $\K[\X]$-linearly independent.  The computation of the dimension follows directly from what we have just showed, as
$$\mbox{\rm ker}\big(\Phi_0\big)_{(1,\ell)}=\K[\X]_{\ell-|\sigma_q-\tau_q|}\cdot F_{q,1}(\T,\X)\oplus\K[\X]_{\ell-|\sigma_{q+1}-\tau_{q+q}|}\cdot F_{q+1,1}(\T,\X).
$$
\end{proof}
\smallskip
\begin{example}
For $(d,u)=(10,3),$ by inspecting the elements of $\cF_0$ given in \eqref{103}, we have that $q=5,$ and
\begin{equation}\label{xx}
\begin{array}{cclcclccl}
|\sigma_q|&=&5,&
|\tau_q|&=&2,&|\sigma_q-\tau_q|=7\\
|\sigma_{q+1}|&=&2,&
|\tau_{q+1}|&=&1& |\sigma_{q+1}-\tau_{q+1}|&=&3.
\end{array}
\end{equation}
So, we actually have that $\dim_{\K}\left(\mbox{\rm ker}(\Phi_0)_{1,\ell}\right)={\ell-5\choose2}+{\ell-1\choose2},$ this number is equal to $\ell^2-5\ell+10$ if $\ell\geq5.$
\end{example}

\begin{theorem}\label{adjj} For $\ell\geq d-2,$ consider a general element in $\mbox{ker}(\Phi_0)_{1,\ell}$  of the form
\begin{equation}\label{ppp}
A_{\ell-|\sigma_q-\tau_q|}(\X)F_{q,1}(\T,\X)+B_{\ell-|\sigma_{q+1}-\tau_{q+1}|}(\X)F_{q+1,1}(\T,\X),
\end{equation}
where
$$
A_{\ell-|\sigma_q-\tau_q|}(\X)=\sum_\ualpha a_\ualpha\,\X^\ualpha, \ \
B_{\ell-|\sigma_{q+1}-\tau_{q+1}|}(\X)=\sum_\ubeta b_\ubeta\,\X^\ubeta,
$$
the sum being over those $\ualpha,\,\ubeta$ such that  $|\ualpha|=\ell-|\sigma_q-\tau_q|,$ and $|\ubeta|=\ell-|\sigma_{q+1}-\tau_{q+1}|.$

Then, \eqref{ppp} belongs to $\mbox{\rm Adj}_\ell\big(\cC_{u,d}\big)$ if and only if $a_\ualpha=0$ for those $\ualpha$ such that  \eqref{pp1} holds, 
and $b_\ubeta=0$ for those $\ubeta$ such that \eqref{pp2} holds.

\end{theorem}
\begin{proof}
Suppose w.l.o.g. that $\sigma_q\geq0,$ and that
$$\begin{array}{ccl}
F_{q,1}(\T,\X)&=&T_0X_0^{\sigma_q}X_2^{-\tau_q}-T_1X_1^{\sigma_q-\tau_q},\\
F_{q+1,1}(\T,\X)&=&T_0X_1^{-\sigma_{q+1}+\tau_{q+1}}-T_1X_0^{-\sigma_{q+1}}X_2^{\tau_{q+1}}
\end{array}
$$
(all the other cases follow mutatis mutandis what follows).  Note that we then have
\begin{equation}\label{numm}
\begin{array}{ccl}
\sigma_q-\sigma_{q+1}&=&d-u,\\
\tau_{q+1}-\tau_q&=&u.
\end{array}
\end{equation}
We expand \eqref{ppp} to get
$$\begin{array}{c}
\left(\sum_\ualpha a_\ualpha\,\X^\ualpha\right)\cdot\left(T_0X_0^{\sigma_q}X_2^{-\tau_q}-T_1X_1^{\sigma_q-\tau_q}\right)+
\left(\sum_\ubeta b_\ubeta\,\X^\ubeta\right)\cdot\left(T_0X_1^{-\sigma_{q+1}+\tau_{q+1}}-T_1X_0^{-\sigma_{q+1}}X_2^{\tau_{q+1}}\right) \\ = \\
T_0\left(\sum_\ualpha a_\ualpha\,\X^\ualpha X_0^{\sigma_q}X_2^{-\tau_q}+\sum_\ubeta b_\ubeta\,\X^\ubeta X_1^{-\sigma_{q+1}+\tau_{q+1}}\right)-T_1\left(\sum_\ualpha a_\ualpha\,\X^\ualpha X_1^{\sigma_q-\tau_q}+\sum_\ubeta b_\ubeta\,\X^\ubeta X_0^{-\sigma_{q+1}}X_2^{\tau_{q+1}}\right).
\end{array}$$
We will apply Proposition \ref{key} to the forms $\sum_\ualpha a_\ualpha\,\X^\ualpha X_0^{\sigma_q}X_2^{-\tau_q}+\sum_\ubeta b_\ubeta\,\X^\ubeta X_1^{-\sigma_{q+1}+\tau_{q+1}}$ and $\sum_\ualpha a_\ualpha\,\X^\ualpha X_1^{\sigma_q-\tau_q}+\sum_\ubeta b_\ubeta\,\X^\ubeta X_0^{-\sigma_{q+1}}X_2^{\tau_{q+1}}$ to see which conditions one should impose on $a_\ualpha$ and $b_\ubeta$ in such a way that these forms define adjoints to $\cC_{u,d}.$ Note first that if
$\X^\ualpha X_0^{\sigma_q}X_2^{-\tau_q}=\X^\ubeta X_1^{-\sigma_{q+1}+\tau_{q+1}}
$ for some $\ualpha,\,\ubeta,$ then this monomial is actually a multiple of $X_0^{\sigma_q}X_1^{-\sigma_{q+1}+\tau_{q+1}}X_2^{-\tau_q},$ which defines a curve which is adjoint to $\cC_{u,d}$ thanks to Proposition \ref{key}, due to the fact that
$$\begin{array}{ccl}
u(-\sigma_{q+1}+\tau_{q+1})-d\tau_q&=&\big(-u\sigma_{q+1}-(d-u)\tau_{q+1}\big)+d\big(\tau_{q+1}-\tau_q\big)\\
&=&-1+du\geq(d-1)(u-1)
\end{array}
$$
(the second equality follows from \eqref{numm} and \eqref{bezz}); and also
$$\begin{array}{ccl}
d\sigma_q+(d-u)(-\sigma_{q+1}+\tau_{q+1})&=&d(\sigma_q-\sigma_{q+1})+u\sigma_{q+1}+(d-u)\tau_{q+1}\\
&=&d(d-u)+1\geq(d-1)(d-u-1),
\end{array}
$$ again thanks to  \eqref{numm} and \eqref{bezz}. This implies that $\sum_\ualpha a_\ualpha\,\X^\ualpha X_0^{\sigma_q}X_2^{-\tau_q}+\sum_\ubeta b_\ubeta\,\X^\ubeta X_1^{-\sigma_{q+1}+\tau_{q+1}}$ (resp. $\sum_\ualpha a_\ualpha\,\X^\ualpha X_1^{\sigma_q-\tau_q}+\sum_\ubeta b_\ubeta\,\X^\ubeta X_0^{-\sigma_{q+1}}X_2^{\tau_{q+1}}$) defines a curve adjoint to $\cC_{u,d}$ if and only if  both
$\sum_\ualpha a_\ualpha\,\X^\ualpha X_0^{\sigma_q}X_2^{-\tau_q}$ and $\sum_\ubeta b_\ubeta\,\X^\ubeta X_1^{-\sigma_{q+1}+\tau_{q+1}}$
(resp. $\sum_\ualpha a_\ualpha\,\X^\ualpha X_1^{\sigma_q-\tau_q}$ and $\sum_\ubeta b_\ubeta\,\X^\ubeta X_0^{-\sigma_{q+1}}X_2^{\tau_{q+1}}$)
define adjoints to $\cC_{u,d}.$
 We analyze these four forms separately, and get the following:
\begin{itemize}
\item for  the $\ualpha$'s, one must have $a_\ualpha=0$ if and only if  the following  holds (the first two from one of the forms,  and the remaining from the other):
\begin{equation}\label{ppr}
\begin{array}{ccl}
u\alpha_1+d(\alpha_2-\tau_q)&<& (d-1)(u-1),\ \mbox{or}\\
d(\alpha_0+\sigma_q)+(d-u)\alpha_1&<&(d-1)(d-u-1)\\ &\mbox{and} & \\
u(\alpha_1+\sigma_q-\tau_q)+d\alpha_2&<& (d-1)(u-1), \ \mbox{or}\\
d\alpha_0+(d-u)(\alpha_1+\sigma_q-\tau_q)&<&(d-1)(d-u-1).
\end{array}
\end{equation}
As $u\tau_q+(d-u)\sigma_q=1,$ we then have that $d\tau_q> u(\tau_q-\sigma_q),$ and $-d\sigma_q<-(d-u)(\sigma_q-\tau_q).$ We then deduce that  \eqref{ppr} is equivalent to
\eqref{pp1}

\item For $\ubeta$'s we must have  $b_\ubeta=0$ if and only if
the following system of inequalities hold:
\begin{equation}\label{prp}
\begin{array}{ccl}
u(\beta_1+\tau_{q+1}-\sigma_{q+1})+d\beta_2&<& (d-1)(u-1),\, \mbox{or}\\
d\beta_0+(d-u)(\beta_1+\tau_{q+1}-\sigma_{q+1})&<&(d-1)(d-u-1)\\ &\mbox{and}&\\
u\beta_1+d(\beta_2+\tau_{q+1})&<& (d-1)(u-1),\, \mbox{or}\\
d(\beta_0-\sigma_{q+1})+(d-u)\beta_1&<&(d-1)(d-u-1)
\end{array}
\end{equation}
Analyzing as before, we get that  \eqref{prp} is equivalent to \eqref{pp2}.
\end{itemize}
\end{proof}
Theorem \ref{kkk} in Section \ref{setup} follows straightforwardly from Theorem \ref{adjj} and Lemma \ref{ddim} above.

\smallskip
We close this section by showing some estimates on the size of the vector spaces involved in these calculations.
\begin{lemma}\label{aux}
For $\ell\geq d-2,$ the set of solutions of  \eqref{pp2} is not empty.
\end{lemma}
\begin{proof}
Note that if one of the two members of the right hand side of \eqref{pp2} is positive, then we would have that either $(\ell-|\sigma_{q+1}-\tau_{q+1}|,0,0)$ or $(0,0,\ell-|\sigma_{q+1}-\tau_{q+1}|)$ is a solution of \eqref{pp2}. So, if there are  no solutions of  this system, then we must have
$$(d-1)(u-1)+u(\sigma_{q+1}-\tau_{q+1})\leq0 \ \mbox{and}\
(d-1)(d-u-1)+d\sigma_{q+1}\leq0.
$$
Adding these two inequalities, we get
$$\begin{array}{ccl}
(d-1)(d-2)&\leq& -u(\sigma_{q+1}-\tau_{q+1})-d\sigma_{q+1}\\
&=&-u\sigma_{q+1}-(d-u)\tau_{q+1}+d(\tau_{q+1}-\sigma_{q+1})\\
&=&-1+d|\tau_{q+1}-\sigma_{q+1}|\\
&\leq &-1+d\,\frac{d}2,
\end{array}
$$
the last inequality holds thanks to Proposition \ref{connection}-viii). This shows that $d$ must be less than or equal to $4$. As we are assuming $1<u\leq \frac{d}2,$ then we are forced to have $u=2$ and $d=4,$ which is impossible as $u$ and $d$ are supposed to be coprime.
\end{proof}
\smallskip
\begin{corollary}
For all $\ell\geq d-2,\,\mbox{\em ker}(\Phi_0)_{(1,\ell)}$ is not contained in $\mbox{\em Adj}_\ell(\cC_{u,d}).$
\end{corollary}
\begin{proof}
From the proof of Lemma \ref{aux}, we deduce straightforwardly that at least one among $X_0^{\ell-|\sigma_{q+1}-\tau_{q+1}|}F_{q+1,1}(\T,X)$ and $X_2^{\ell-|\sigma_{q+1}-\tau_{q+1}|}F_{q+1,1}(\T,X)$  does not belong to $\mbox{ Adj}_\ell(\cC_{u,d}).$
\end{proof}
\smallskip
We finish by showing a rough estimate on the number $\nu_{u,d},$ the dimension of the quotient $\mbox{\rm ker}(\Phi_0)_{(1,\ell)}\, /\,\mbox{\rm Adj}_\ell(\cC_{u,d})\cap\mbox{\rm ker}(\Phi_0)_{(1,\ell)}$ for $\ell\geq d-2.$
\begin{proposition}\label{arriba}
$$\nu_{u,d}\leq d^2-6d+6.$$
\end{proposition}
\begin{proof}
Suppose again w.l.o.g. that $\sigma_q\geq0.$ We use Lemma \ref{diof} to bound the number of solutions of \eqref{pp1} and \eqref{pp2}, to get that
$$\begin{array}{ccl}
\nu_{u,d}&\leq & 2(d-1)(d-2)+d\tau_q-(d-u)(\sigma_q-\tau_q)+u(\sigma_{q+1}-\tau_{q+1})+d\sigma_{q+1}\\
&=&2(d-1)(d-2)+d(\tau_q-\sigma_q+\sigma_{q+1}-\tau_{q+1})+u\sigma_q+(d-u)\tau_q+u\sigma_{q+1}+(d-u)\tau_{q+1}\\
&=&2(d-1)(d-2)-d^2+2\\
&=& d^2-6d+6.
\end{array}$$
\end{proof}
\smallskip
We will see at the end of the following section that for some family of examples, the number $\nu_{u,d}$ grows quadratically in $d^2,$ hence one should regard the bound in Proposition \ref{arriba} as asymptotically optimal.
\bigskip
\section{Further examples}\label{examples}
We conclude this paper by working out  a couple of examples to show how all the elements in the resolution of $\mbox{Rees}({\mathcal I})$ can be computed straightforwardly by using the results obtained in the previous sections. We will start by studying with detail the case $(d,u)=(14,3).$ Here, we have that the ordinary Euclidean algorithm gives:
$$\begin{array}{ccl}
11&=&3\cdot 3+2\\
3&=&1\cdot 2+1\\
2&=&2\cdot 1+0
\end{array}
$$
and so
$p=4$, $q=6$, $\{a_n\}=\{11,\,3,\,2,\,1\,\, 0\},\,\{q_m\}=\{3,\,1,\,2\}$, $\{ m_l \}=\{ 1,\,4,\,5,\,7,\,8\}$.
Its SERS  is
$$\begin{array}{ccl}
(b_1,c_1)&=&(11,3)\\
(b_2, c_2)&=&(8,3)\\
(b_3, c_3)&=&(5,3)\\
(b_4, c_4)&=&(3,2)\\
(b_5, c_5)&=&(2,1)\\
(b_6,c_6)&=&(1,1)\\
(b_7,c_7)&=&(1,0),
\end{array}
$$
and its Extended SERS:
$$\begin{array}{ccl}
(\sigma_1,\tau_1,\alpha_1,\beta_1)&=&(0,1,1,0)\\
(\sigma_2,\tau_2,\alpha_2,\beta_2)&=&(-1,1,1,0)\\
(\sigma_3,\tau_3,\alpha_1,\beta_1)&=&(-2,1,1,0)\\
(\sigma_4,\tau_4,\alpha_1,\beta_1)&=&(1,0,-3,1)\\
(\sigma_5,\tau_5,\alpha_1,\beta_1)&=&(-3,1,4,-1)\\
(\sigma_6,\tau_6,\alpha_1,\beta_1)&=&(-7,2,4,-1)\\
(\sigma_7,\tau_7,\alpha_1,\beta_1)&=&(4,-1,-11,3).
\end{array}
$$
Observe that $\sigma_q=\sigma_6=-7<0$. Thus,  $\cF_0$ is a a minimal system of generators of $\mbox{ker}(\Phi_0)$ consisting of $q+2=8$ polynomials. Computed explicitly via \eqref{fgen}, we obtain:
$$
\begin{array}{ccl}
F_{1,11}(\T,\X)&=&T_0^{11}X_2-T_1^{11}X_1\\
F_{2,8}(\T,\X)&=&T_0^8X_0X_2-T_1^8X_1^2\\
F_{3,5}(\T,\X)&=&T_0^5X_0^2X_2-T_1^5X_1^3\\
F_{4,3}(\T,\X)&=& T_0^3X_1-T_1^3X_0\\
F_{5,2}(\T,\X)&=& T_0^2X_0^3X_2-T_1^2X_1^4\\
F_{6,1}(\T,\X)&=& T_0X_0^7X_2^2-T_1X_1^9\\
F_{7,1}(\T,\X)&=& T_0X_1^5-T_1X_0^4X_2\\
F_{8,0}(\T,\X)&=& X_1^{14}-X_0^{11}X_2^3.
\end{array}
 $$
Their  bidegrees are
$$
\{ \mbox{bideg}(F_{n,b_n}) \}=\{ (11,1),(8,2),(5,3),(3,1),(2,4),(1,9),(1,5),(0,14) \}
$$
Now we turn to the computation of a basis of $\mbox{syz}(\cF_0).$ By Theorem \ref{phi1} and \eqref{2en1},\,\eqref{rrest}, we have that the family $\cF_1 \subset S^8$ is made by the following $12$ syzygies:
 \[\begin{array}{ll}
\bs_{1,2}=X_0\e_1-T_0^3\e_2-T_1^8X_1\e_4,&
\bs_{1,4}=X_1\e_1-T_0^8X_2\e_4-T_1^3\e_2 \\
\bs_{2,3}=X_0\e_2-T_0^3 \e_3-T_1^5X_1^2\e_4,&
\bs_{2,4}=X_1\e_2-T_0^5X_0X_2\e_4-T_1^3\e_3\\
\bs_{3,4}=X_1\e_3-T_0^2X_0^2X_2\e_4-T_1^3\e_5,&
\bs_{3,5}=X_0\e_3-T_0^3\e_5-T_1^2X_1^3\e_4 \\
\bs_{4,5}=X_0^3X_2\e_4-T_0X_1\e_5-T_1^2\e_7,&
\bs_{4,7}=X_1^4\e_4-T_0^2\e_7-T_1X_0\e_5\\
\bs_{5,6}= X_0^4X_2\e_5-T_0\e_6-T_1X_1^4\e_7,&
\bs_{5,7}=X_1^5\e_5-T_0X_0^3X_2\e_7-T_1\e_6\\
\bs_{6,7}=X_1^5\e_6-X_0^7X_2^2\e_7-T_1\e_8,&
\bs_{7,8}=X_1^9\e_7-T_0\e_8-X_0^4X_2\e_6,
\end{array}
\]
with $
\{ \mbox{bideg}(s_{n,m_{l(n)}}) = \mbox{bideg}(s_{n,\sigma(n)} ) \}_{n=1,\ldots, 5}=\{ (11,2),(8,3),(5,4),(3,5),(2,9) \},
$
and $\mbox{bideg}(\bs_{6,7})=\mbox{bideg}(\bs_{7,8})=(1,14)$.

 We can also compute the elements of family $\cF_2 \subset S^{12}$ via Proposition \ref{aiecomae}. It consists of the following syzygies:
\[\begin{array}{l}
\bs_{1,2,4}= X_1\e_{1,2}-X_0\e_{1,4}+T_0^3\e_{2,4}-T_1^3\e_{2,3}  \\
\bs_{2,3,4}= X_1\e_{2,3}-X_0\e_{2,4}+T_0^3\e_{3,4}-T_1^3\e_{3,5}\\
\bs_{3,4,5}=X_0\e_{3,4}-X_1\e_{3,5}+T_0^2\e_{4,5}-T_1^2\e_{4,7} \\
\bs_{4,5,7}= X_1^4\e_{4,5}-X_0^3X_2\e_{4,7}+T_0\e_{5,7}-T_1\e_{5,6}\\
\bs_{5,6,7}= X_1^5\e_{5,6}-X_0^4X_2\e_{5,7}+T_0\e_{6,7}-T_1\e_{6,8},
\end{array}
\]
with $\{ \mbox{bideg}(s_{n,\rho(n),l(n)} )\} = \{ (11,3),(8,4),(5,5),(3,9),(2,14) \}.$
So, we get a whole description of the resolution of $\mbox{Rees}({\mathcal I})$ for this case.
\par Let us turn now to the case of adjoints.
In order to determine $\dim\big(\mbox{Adj}_\ell(\cC_{3,14})\big)$ and  $\dim\big(\mbox{ker}(\Phi_0)_{1,\ell}\big)$, for $\ell\geq 12
$, we compute $q=6,\,\sigma_6=-7$, $\tau_6=2$, $\sigma_7=4$, $\tau_7=-1,$
$|\sigma_6-\tau_6|=9$, and $|\sigma_7 -\tau_7|=5$.
Then, we have
$$\dim_\K\left(\mbox{Adj}_\ell(\cC_{3,14})\right)= (\ell+2)(\ell+1)-156 =\ell^2+3\ell-154,$$
and
$\dim_{\K}\left(\mbox{ker}(\Phi_0)_{1,\ell}\right)= {\ell-7\choose 2}+{\ell-3\choose 2}=\ell^2-11\ell+34.
$
To make  $\nu_{3,14}$ explicit, we have to compute the number of solutions of 
$$
3\alpha_1+14\alpha_2<-2  \quad \mbox{or}\ \
14\alpha_0+11\alpha_1<31,
$$
and also the number of solutions of
$$ 3\beta_1+14\beta_2<11\quad \mbox{or}\ \
14\beta_0+11\beta_1<74,
$$
with $(\alpha_0,\alpha_1,\alpha_2)$,  $(\beta_0,\beta_1,\beta_2)\in\N^3$, $\alpha_0+\alpha_1+\alpha_2=\ell-9$,
 $\beta_0+\beta_1+\beta_2=\ell-5$.

 The first of the four inequalities has no integer solutions. For the other three, there are $6$ solutions for $14\alpha_0+11\alpha_1<31,$ one per each of the values $0,14,28,11,22,25$;
 $4$ solutions for  the values of $3\beta_1+14\beta_2<11$ corresponding to $0,3,6,9;$ and $23$ solutions of $14\beta_0+11\beta_1<74,$ for the values    $0,11,14,22,25,28,33,36,39,42,44,47,50,53,$ $55,56,58,61,64,66,67,69,70.$
\par
So, we have, for $\ell\geq12,$
 $$\dim_\K\left(\mbox{\rm ker}(\Phi_0)_{(1,\ell)}\, /\,\mbox{\rm Adj}_\ell(\cC_{3,14})\cap\mbox{\rm ker}(\Phi_0)_{(1,\ell)} \right)=\nu_{3,14}=4+6+23=33.
$$
\par
Let us consider now the case where $u$ is fixed, $u=2$, and $d$ an integer coprime with $2,$ i.e. $d=2k-1$ with $k\geq3.$ For this case, we have that the ordinary Euclidean sequence is
$$\begin{array}{cclcl}
2k-3&=&(k-2)\cdot 2&+&1\\
2&=&2\cdot 1&+&0.
\end{array}
$$
Hence,
$p=3$, $q=k$,
$\{a_n\}=\{2k-3,\,2,\,1\,\, 0\},\,\{q_m\}=\{k-2,\,2\}$, and $\{ m_l \}=\{ 1,\,k-1,\,k+1,\,k+2\}$.
The SERS  associated to this data is
$$\begin{array}{lcl}
(b_1,c_1)&=&(2k-3,2)\\
(b_2, c_2)&=&(2k-5,2)\\
& \vdots &\\
(b_n, c_n)&=&(2k-3-2(n-1),2)=(2(k-n+1)-3,2)=(2(k-n)-1,2)\\
& \vdots & \\
(b_{k-2}, c_{k-2})&=&(3,2)\\
(b_{k-1},c_{k-1})&=&(2,1)\\
(b_k,c_k)&=&(1,1)\\
(b_{k+1},c_{k+1})&=&(1,0),
\end{array}
$$
and its Extended SERS is
$$\begin{array}{lcl}
(\sigma_1,\tau_1,\alpha_1,\beta_1)&=&(0,1,1,0)\\
(\sigma_2,\tau_2,\alpha_2,\beta_2)&=&(-1,1,1,0)\\
 & \vdots & \\
(\sigma_n,\tau_n,\alpha_n,\beta_n)&=&(-(n-1),1,1,0)\\
 & \vdots & \\
(\sigma_{k-2},\tau_{k-2},\alpha_{k-2},\beta_{k-2})&=&(-(k-3),1,1,0)\\
(\sigma_{k-1},\tau_{k-1},\alpha_{k-1},\beta_{k-1})&=&(1,0,-(k-2),1)\\
(\sigma_k,\tau_k,\alpha_k,\beta_k)&=&(k-1,-1,-(k-2),1)\\
(\sigma_{k+1},\tau_{k+1},\alpha_{k+1},\beta_{k+1})&=&(-(k-2),1,2k-3,-2).
\end{array}
$$
Observe that $\sigma_q=\sigma_k>0$. Thus, the family $\cF_0$ is made by the following $k+2$ polynomials:
$$
\begin{array}{lcl}
F_{1,2k-3}(\T,\X)&=&T_0^{2k-3}X_2-T_1^{2k-3}X_1\\
&\vdots & \\
F_{n,2(k-n)-1}(\T,\X)&=&T_0^{2(k-n)-1}X_0^{n-1}X_2-T_1^{2(k-n)-1}X_1^n\\
&\vdots &\\
F_{k-2,3}(\T,\X)&=&T_0^3X_0^{k-3}X_2-T_1^3X_1^{k-2}\\
F_{k-1,2}(\T,\X)&=& T_0^2X_1-T_1^2X_0\\
F_{k,1}(\T,\X)&=& T_0X_1^k-T_1X_0^{k-1}X_2\\
F_{k+1,1}(\T,\X)&=& T_0X_0^{k-2}X_2-T_1X_1^{k-1}\\
F_{k+2,0}(\T,\X)&=& X_0^{2k-3}X_2^2-X_1^{2k-1}
\end{array}
 $$
with bidegrees
$$
\{ (2k-3,1),(2k-5,2),\dots, (2(k-n)-1,n),\dots, (3,k-2),(2,1),(1,k),(1,k-1),(0,2k-3) \}.
$$
The elements of $\cF_0$ have already been worked out in our previous paper \cite[Example 3.5]{CD13b}. 
In order to describe the families $\cF_1$ and $\cF_2,$ we note that
$$
m_{\ell(n)}=\left\{\begin{array}{ccl}
k-1&\,\mbox{ if }\,&1\leq n <k-1\\
k+1 &\,\mbox{ if }\,&k-1 \leq n \leq k \\
\end{array}\right.
$$
and
$$ \rho(n)=\left\{\begin{array}{ccl}
n+1&\,\mbox{ if }\,& 1\leq n <k-2,\\
k+1 & \,\mbox{ if }\,& n =k-2,\\
k & \,\mbox{ if }\,& n =k-1,\\
k+2 & \,\mbox{ if }\,& n =k.\\
\end{array}\right.
$$
So, $\cF_1 \subset S^{k+2}$ is made by
 \[\begin{array}{l}
\bs_{n,n+1}=X_0\e_n-T_0^2\e_{n+1}-T_1^{2(k-n+1)-1}X_1^n\e_{k-1},   \\
\bs_{n,k-1}=X_1\e_n-T_0^{2(k-n+1)-1}X_0^{n-1}X_2\e_{k-1}-T_1^2\e_{n+1} ,
\mbox{ for }1\leq n <k-2,\\
\bs_{k-2,k-1}=X_1\e_{k-2}-T_0X_0^{k-3}X_2\e_{k-1}-T_1^2\e_{k+1},   \\
\bs_{k-2,k+1}=X_0\e_{k-2}-T_0^2\e_{k+1}-T_1X_1^{k-2}\e_{k-1} , \\
\bs_{k-1,k}=X_1^{k-1}\e_{k-1}-T_0\e_k-T_1X_0\e_{k+1},   \\
\bs_{k-1,k+1}=X_0^{k-2}X_2\e_{k-1}-T_0X_1\e_{k+1}-T_1\e_k ,\\
\bs_{k,k+1}=X_0^{k-2}X_2\e_k-X_1^k\e_{k+1}-T_1\e_{k+2} ,\\
\bs_{k+1,k+2}=X_0^{k-1}X_2\e_{k+1}-T_0\e_{k+2}-X_1^{k-1}\e_k ,\\
\end{array}
\]
with
 \[\begin{array}{l}
\mbox{bideg}(\bs_{n,n+1})= \mbox{bideg}(\bs_{n,k-1})=(2(k-n)-1,n+1)$, for $1\leq n <k-2,\\
 \mbox{bideg}(\bs_{k-2,k-1})= \mbox{bideg}(\bs_{k-2,k+1})=(3,k-1),\\
  \mbox{bideg}(\bs_{k-1,k})= \mbox{bideg}(\bs_{k-1,k+1})=(2,k),\\
  \mbox{bideg}(\bs_{k,k+1})= \mbox{bideg}(\bs_{k+1,k+2})=(1,2k-1).
  \end{array}
\]
In addition, $\cF_2 \subset S^{2k}$ consists of the following $k$ elements
 \[\begin{array}{l}
\bs_{n,n+1,k-1}=X_1\e_{n,n+1}-X_0\e_{n,k-1}+T_0^2\e_{n+1,k-1}-T_1^2\e_{n+1,n+2},
\mbox{ for }1\leq n <k-3,\\
\bs_{k-3,k-2,k-1}=X_1\e_{k-3,k-2}-X_0\e_{k-3,k-1} +T_0^2e_{k-2,k-1}-T_1^2\e_{k-2,k+1},   \\
\bs_{k-2,k+1,k-1}=X_0\e_{k-2,k-1}-X_1\e_{k-2,k+1}+T_0\e_{k-1,k+1}-T_1\e_{k-1,k} ,\\
\bs_{k-1,k.k+1}=X_0^{k-2}X_2\e_{k-1,k}-X_1^{k-1}\e_{k-1,k+1}+T_0\e_{k,k+1}-T_1\e_{k,k+2}.
\end{array}
\]
with 
\[\begin{array}{l}
\mbox{bideg}(\bs_{n,n+1,k-1} )=(2(k-n)-1,n+2)$, for $1\leq n <k-2,\\
 \mbox{bideg}(\bs_{k-2,k+1,k-1})=(3,k),\\
  \mbox{bideg}(\bs_{k-1,k,k+1})= \mbox{bideg}(\bs_{k-1,k+1})=(2,2k-1).
  \end{array}
\]
This concludes the computation of the elements in the minimal resolution of $\mbox{Rees}({\mathcal I})$ for this case. Let us finish the paper by showing that $\nu_{2,d}$ grows quadratically with $d.$ Indeed, as in this case we have $\tau_q=\tau_k=-1$, and $|\sigma_q-\tau_q|=k,$ then \eqref{pp1} reads as follows:
\begin{equation}\label{specc}
\left\{\begin{array}{lllr}
2\alpha_1+(2k-1)\alpha_2&<&-1&\,\mbox{or}\\
(2k-1)\alpha_0+(2k-3)\alpha_1&<&2k^2-9k+8.&
\end{array}
\right.
\end{equation}
The first inequality has clearly no positive solutions. Now, note that if 
\begin{equation}\label{p1p}
(2k-1)(\alpha_0+\alpha_1)<2k^2-9k+8,
\end{equation} then
$(\alpha_0,\,\alpha_1)$ satisfies the second inequality of \eqref{specc} above.
As $\frac{2k^2-9k+8}{2k-1}=k-4+\frac{4}{2k-1},$ we get from \eqref{p1p} that if
$\alpha_0+\alpha_1\leq k-4,
$ 
then $(\alpha_0,\,\alpha_1)$ satisfies the second inequality of \eqref{specc}. The number of solutions of the last inequality in $\N^2$ can be computed easily:  it is equal to ${k-2\choose 2}=\frac{(k-2)(k-3)}{2}={\mathcal{O}}\big(d^2/8\big).$ This shows that $\nu_{2,d}$ grows at least as $d^2/8$ in this family of examples.

\end{document}